\documentclass[11pt, reqno]{amsart}

\usepackage{amsmath}
\usepackage{extarrows}
\usepackage{thmtools}
\usepackage{amsthm}
\usepackage{amssymb}
\usepackage{mathtools}
\usepackage{graphicx}
\usepackage{indentfirst}
\usepackage{cancel}
\usepackage{bm}
\usepackage{enumerate}
\usepackage{xcolor}
\usepackage{comment}

\DeclareMathAlphabet{\mathpzc}{OT1}{pzc}{m}{it}
\usepackage[pdfauthor={},
    pdftitle={},%
    pdftex]{hyperref}
\hypersetup{
colorlinks,
citecolor=red,
filecolor=black,
linkcolor=blue,
urlcolor=blue
}
\usepackage{caption}
\usepackage{subcaption}
\usepackage{float}
\usepackage{tikz-cd}
\usepackage[all,cmtip]{xy}
\usepackage[mathscr]{euscript}
\allowdisplaybreaks
\makeatother
\usepackage[left=3.00cm, right=3.00cm, top=2.00cm, bottom=3.00cm]{geometry}
\usepackage{adjustbox}
\newcommand{\C}{\mathbb{C}}
\newcommand{\Q}{\mathbb{Q}}
\newcommand{\R}{\mathbb{R}}

\newcommand{\Pa}{\mathcal{P}}
\newcommand{\Z}{\mathbb{Z}}

\newcommand{\Om}{\Omega}

\DeclareMathOperator{\ad}{\mathsf{ad}}

\newcommand{\compcent}[1]{\vcenter{\hbox{$#1\circ$}}}
\newcommand{\comp}{\mathbin{\mathchoice
{\compcent\scriptstyle}{\compcent\scriptstyle}
{\compcent\scriptscriptstyle}{\compcent\scriptscriptstyle}}} 
\usepackage{relsize}
\newcommand\cca[1]{{%
\ooalign{\raisebox{-.4ex}{\larger[4]$\circlearrowright$}\cr
\hidewidth$\,#1$\hidewidth}}}

\newtheoremstyle{mystyle}{3pt}{3pt}{}{}{\bfseries}{.}{.5em}{}

\numberwithin{equation}{section}
\theoremstyle{plain}
\newtheorem{thm}{Theorem}[section]
\newtheorem{lemma}[thm]{Lemma}
\newtheorem{coro}[thm]{Corollary}
\newtheorem{customthm}{Theorem}

\newtheorem*{thm*}{Theorem}

\newtheorem{prop}[thm]{Proposition}
\theoremstyle{definition}
\newtheorem{defi}[thm]{Definition}
\newtheorem{remark}[thm]{Remark}
\newtheorem{eg}[thm]{Example}

\AddToHook{cmd/appendix/before}{%
\setcounter{thm}{0}%
}

\begin{document}
\title[Generalized brace product]{On the generalized brace product: relation to $H$-splitting of loop space fibrations \& the $J$-homomorphism}
\author[Basu, Bhowmick, Samanta]{Somnath Basu, Aritra Bhowmick, Sandip Samanta}
\address{Department of Mathematics and Statistics, Indian Institute of Science Education and Research Kolkata,
Mohanpur, West Bengal, India -- 741 246}
\email{somnath.basu@iiserkol.ac.in, avowmix@gmail.com, ss18ip021@iiserkol.ac.in}
\keywords{Fibration, $H$-splitting, Brace Product, Whitehead Product, Samelson Product, $J$-homomorphism}
\subjclass[2020]{55Q15, 55R15, 55R05, 55R25, 55Q35, 55Q50, 55P10, 55P35, 55P45, 55P62}

\begin{abstract}
   Given a fibration $F \hookrightarrow E \rightarrow B$ with a homotopy section $s: B \rightarrow E$, James introduced a binary product $\left\{, \right\}_s: \pi_i B \times \pi_j F \rightarrow \pi_{i+j-1} F$, called the brace product, which was later generalized by Yoon.
   We show that the vanishing of this generalized brace product is the precise obstruction to the $H$-splitting of the loop space fibration, i.e., $\Omega E \simeq \Omega B \times \Omega F$ as $H$-spaces. Using rational homotopy theory, we show that for rational spaces, the vanishing of the generalized brace product coincides with the vanishing of the classical James brace product, enabling us to perform the relevant computations. In addition, the notion of $J$-homomorphism is generalized and connected to the generalized brace product. Among the applications, we characterize the homotopy types of certain fibrations, including sphere bundles over spheres. 
\end{abstract}

\maketitle

\section{Introduction}
A fibration $p: E \rightarrow B$ with homotopy fibre $F$, denoted by $F\overset{i}{\hookrightarrow} E\overset{p}{\to} B$, is said to be \emph{decomposable} (or \textit{trivial}) if $E$ is homotopy equivalent to the product space $B \times F$. A principal $G$-bundle $p: E \rightarrow B$ equipped with a section $s: B \rightarrow E$ is decomposable since $E$ is homeomorphic to $B\times G$. However, for a general fibration, merely having a section is a much weaker condition than the fibration being decomposable. We refer to a map $s: B \rightarrow E$ as a \emph{homotopy section} if $p \comp s$ is homotopic to the identity map of the base space $B$. 
In \cite{James_decomposable} James introduced a binary product $\left\{ , \right\}_s : \pi_* (B) \times \pi_* (F) \rightarrow \pi_* (F)$ for a fibration $p$ equipped with a homotopy section $s$. We shall call this product the \emph{James brace product} (or simply, the \textit{brace product}). James demonstrated that the brace product is an obstruction to certain fibrations being decomposable. Kallel and Sjerve \cite{Kallel-Sjerve} established a connection between the brace product and certain differentials in the Serre spectral sequence associated with a fibration having a section. In this article, we prove the following result (Theorem~\ref{phi_alg_map}).

\begin{customthm}
    Suppose $F\overset{i}{\hookrightarrow} E\overset{p}{\to} B$ is a fibration with a homotopy section $s:B\to E$. Then the James brace product $\{,\}_s$ vanishes for the given fibration if and only if the natural map $(\varphi_{s})_*:=(\Omega s\cdot\Omega i)_*:\pi_*(\Omega B\times\Omega F)\to \pi_*(\Omega E)$ is an algebra isomorphism with ring structure given by the Samelson product.
\end{customthm}

It is well known \cite{Eckmann-Hilton} that if a fibration $p: E \rightarrow B$ admits a homotopy section $s: B \rightarrow E$, then the map $\varphi_s: \Omega B \times \Omega F \rightarrow \Omega E$ is a homotopy equivalence. As any based loop space $\Omega Y$ of $Y$ is an $H$-space, a natural question arises: is $\varphi_s$ an $H$-space map? If so, then call $\varphi_s$ an \emph{$H$-splitting}. In \cite{Kallel-Sjerve}, the authors identified a certain space level analog of the James brace product as the obstruction to $\varphi_s$ being an $H$-splitting. Recently, Liu \cite{Liu} studied the same question and obtained some sufficient conditions. 

To answer the above question completely, we employ the \textit{generalized brace product} introduced by Yoon \cite[Def 3.3]{Yoon-DecomposabilityOfEvaluationFibration} for the groups of homotopy classes of maps from the suspension spaces. Recall, given any based space $X$ and its reduced suspension $\Sigma X$, the set $[\Sigma X, Z]$, of base point preserving homotopy classes of maps $\Sigma X \rightarrow Z$, forms a group.  With this in mind, for a fibration $p: E \rightarrow B$ with homotopy section $s: B \rightarrow E$, and spaces $X, Y$, the \emph{generalized brace product} (Definition~\ref{defn:genBraceProduct}) is defined as a binary operation
\[\left\{ , \right\}_s : [\Sigma X, B] \times [\Sigma Y, F] \rightarrow [\Sigma (X \wedge Y), F],\]
where $X \wedge Y$ represents the smash product of $X$ and $Y$. In \S \ref{sec:genBraceProductProperties}, we establish several basic properties of the generalized brace product and how they change in pullback and product fibration.

In the context of $H$-splitting of the $\varphi_s$ map, we establish the following result (Theorem~\ref{generalized_brace-splitting}).

\begin{customthm}
    The map $\varphi_s$ is an $H$-splitting if and only if the generalized brace product vanishes for all spaces $X, Y$.
\end{customthm}

{As an application of the above theorem, we conclude the $H$-splitting of the based loop space of the $m$-th free loop space fibration of $S^2$ (Example~\ref{Eg:FreeLoopSpaceFibrationOfS2}). Example \ref{eg:LieGroup} provides another instance of $H$-splitting in the context of simply connected compact Lie groups.} Remarkably, the vanishing of the generalized brace product is significantly more stringent than the vanishing of the brace product as studied in James' initial research \cite{James_decomposable} or in \cite{Kallel-Sjerve}. To illustrate this point further, we provide an example (Example~\ref{eg:brace_vanish_but_not_genaralized}) of a fibration with a section exhibiting a vanishing brace product but a non-vanishing \emph{generalized} brace product. This gives a counter-example to Lemma 2.12 of Kallel-Sjerve \cite{Kallel-Sjerve}, which states that ``\textit{if the brace product vanishes identically, then $\varphi_s$ is an $H$-splitting}''. Their lemma remains true in some special cases, such as when both the base and fibre of a fibration are finite wedges of spheres (Corollary~\ref{base_fibre_product-wedge_sphere}). 
Moreover, in the case of a sphere bundle over a sphere, the vanishing of the brace product implies more (Corollary~\ref{brace_in_sphere_over_sphere}).
\begin{customthm}
    Suppose $S^m\overset{i}{\hookrightarrow} E\xrightarrow{p}S^n$ is a fibration. Then $E\simeq S^n\times S^m$ if and only if there exists a homotopy section $s$ such that the James brace product $\{\mathrm{Id}_{S^n},\mathrm{Id}_{S^m}\}_s=0$.
\end{customthm}

We have provided a counter-example (Example \ref{eg:Husemollar_counter}) to Proposition 6.4 of Husem\"{o}ller \cite[p. 226]{Husemoller-FibreBundles} stating that ``\textit{the image of the $J$-homomorphism is contained inside the image of the suspension homomorphism of sphere}''. Proposition \ref{HusemollerRectified} is a rectification of that. Alongside this, we provide computations for the rationalized version of the generalized brace product in certain cases and prove the decomposability of the corresponding rational spaces. This leads to the following result (Theorem~\ref{cor:JamesBraceVanishImpliesRationalEquiv}). {The converse of this theorem is not true though (see Remark~\ref{remark:ConverseOfRationalDecomposition}).}

\begin{customthm}
    Suppose $\Sigma F_\Q \xrightarrow{i_\Q} E_\Q \xrightarrow{p_\Q} \Sigma B_\Q$ is the rationalized fibration with section $s_\Q$ of a fibration $\Sigma F \xrightarrow{i} E \xrightarrow{p} \Sigma B$ with section $s$, where $\Sigma F, E,$ and $\Sigma B$ are simply connected finite CW complexes. If the James brace product $\{,\}_{s_\Q}$ vanishes, then $E$ and $\Sigma B\times \Sigma F$ have the same rational homotopy type.
\end{customthm}

Finally, we go on to explore the relationship between the generalized brace product and a generalized concept of the classical $J$-homomorphism in \S\ref{Sec:J-HomoAndBrace}. In particular, we prove the following result (Theorem~\ref{BraceAndJ-homoInSuspensionOverSuspension}).
\begin{customthm}
    Suppose $\Sigma F \overset{i}{\hookrightarrow} E \rightarrow \Sigma B$ is a fibration characterized by a map $\rho: B\to Map_\bullet(\Sigma F, \Sigma F)$, which maps the base point of $B$ to $\mathrm{Id}_{\Sigma F}$. Then, given any base point preserving self homotopy equivalence $\varphi: \Sigma F \rightarrow \Sigma F$, we have 
    $$\left\{ \mathrm{Id}_{\Sigma B}, \varphi \right\}_s = J[\varphi \comp \varepsilon] -J[\rho],$$ 
    where $s$ is {the choice of a canonical `$\infty$-section'} $s$, and $\varepsilon : B \rightarrow Map_\bullet(\Sigma F, \Sigma F)$ is given by $\varepsilon(b) = \mathrm{Id}_{\Sigma F}$. Moreover, if $B$ is a suspension, then $\left\{ \mathrm{Id}_{\Sigma B}, \varphi \right\}_s = -J[\rho]$.
\end{customthm}
The above Theorem, on the one hand, helps us to prove Lemma 1 (see Lemma~\ref{Milnor-lemma}) of Milnor \cite{Milnor-OnWhiteheadJHomo}, no proof of which could be located in the existing literature. On the other hand, this generalizes Lemma 1 of St\"ocker \cite{Stocker-NoteOnQuasi}, which is a consequence of Theorems 2 and 3 there. Moreover, this leads to the following result (Theorem \ref{thm:sphereBundleRationallySplits}).
\begin{customthm}
    Consider a fibre bundle $S^q \hookrightarrow E \rightarrow S^n$ with structure group $SO(q + 1)$. If the fibration admits a homotopy section $s: S^n\to E$, then $E$ has the same rational homotopy type as $S^q \times S^n$.
\end{customthm}

\medskip

\noindent \textsc{Organization of the paper.} The article has the following structure. In \S\ref{sec:background}, we discuss the necessary background, including the Whitehead and the Samelson product on the homotopy groups of spaces. In \S\ref{sec:genBraceProductProperties}, we introduce the generalized brace product and derive some of its basic properties. In \S\ref{sec:vanishingOfBrace}, we present the main results of this article. In particular, we discuss the consequences and provide equivalent criteria for the vanishing of both the James and the generalized brace products. Following that, in \S\ref{sec:genBraceLocalization}, we discuss various applications of these results in the context of localization. Lastly, in \S\ref{Sec:J-HomoAndBrace}, we introduce the generalized $J$-homomorphism and relate it to a certain generalized brace product. In appendix~\ref{sec:appendix}, we provide proof of a proposition, which might be known to experts.\medskip
    
\noindent \textsc{Notation.} All spaces are assumed to be CW-complexes equipped with a base point, and the maps will be base point preserving continuous maps unless otherwise mentioned. Any homotopy of such a map will be assumed to be base point preserving as well. By a standard abuse of notation, $[\mathcal{P},\mathcal{Q}]$ denotes the collection of homotopy classes of maps from $\mathcal{P}$ to $\mathcal{Q}$ if $\mathcal{P}, \mathcal{Q}$ are spaces, and denotes the (generalized) Whitehead product if $\mathcal{P}, \mathcal{Q}$ are (homotopy classes of) maps. Also, we shall use the same symbol to denote a map and its homotopy class. The reduced cone and reduced suspension of $X$ shall be denoted by $CX$ and $\Sigma X$, respectively. For any space $X, Y$, we denote the bijection $\ad: [\Sigma X, Y] \rightarrow [X, \Omega Y]$ given by the $\Sigma$-$\Omega$ adjunction for arbitrary spaces $X$ and $Y$. We will adopt additive notation for the (potentially non-abelian) group $[\Sigma X, Y]$, and specifically denote the identity element as $0$. The concatenation operation on any based loop space $\Omega X$ will be denoted by ``$\cdot$''. We denote two homeomorphic (resp. homotopy equivalent) spaces $X$ and $Y$ by $X \cong Y$ (resp. $X \simeq Y$), while two homotopic maps $f$ and $g$ will be denoted by $f \simeq g$.\\

\noindent \textsc{Acknowledgments.} The authors would like to thank Sadok Kallel for initial communication and clarification regarding the article \cite{Kallel-Sjerve}. They are also grateful to Stephen Theriault for his valuable comments on the final version of this article, and to Thomas Schick for his helpful suggestion concerning Theorem \ref{thm:sphereBundleRationallySplits}. The authors deeply appreciate the anonymous reviewer's careful reading and thoughtful comments, which have helped strengthen the revised version of this work. The first author appreciates the support received through the SERB MATRICS grant MTR/2017/000807. The second author acknowledges the support provided by the NBHM Post-doctoral fellowship, while the third author is thankful for the PMRF fellowship.

\section{Background} \label{sec:background}
We recall some basic notions regarding binary operations on homotopy groups and refer to \cite{Whitehead_elements} as a general reference for this section.

\begin{defi}[Whitehead product]
    Suppose $f:(I^k,\partial I^k)\to (X,x_0)$ and $g:(I^l,\partial I^l)\to (X,x_0)$ are continuous maps representing elements $[f]\in \pi_k(X,x_0)$ and $[g]\in \pi_l(X,x_0)$. The \textit{Whitehead product} of $[f],[g]$ is denoted by $[f,g]\in \pi_{k+l-1}(X,x_0)$ and defined as the homotopy class of the following map
    \begin{align*}
        S^{k+l-1}=\partial (I^k\times I^l)=(I^k\times \partial I^l)\cup &(\partial I^k\times I^l)\xrightarrow{f\cup g} X
    \end{align*}
\end{defi}

As indicated by the notation, $[\cdot,\cdot]$ is also called the \emph{Whitehead bracket}. It forms a graded quasi-Lie algebra, i.e., $[\cdot,\cdot]$ is an operation of degree $-1$, graded commutative ($[f,g]=(-1)^{|f||g|}[g,f]$), and satisfies the (graded) Jacobi identity for higher homotopy groups
$$(-1)^{|f||h|}[[f,g],h]+(-1)^{|g||f|}[[g,h],f]+(-1)^{|h||g|}[[h,f],g]=0,$$
where $|f|, |g|, |h|\geq 1$. Initially, it was well-known for $|f|, |g|, |h|\geq 2$, but later Hilton \cite{Hilton-JacobiIdentity} appropriately extended it for $|f|, |g|, |h|\geq 1$, using additive notation for the (possibly non-abelian) fundamental group. Moreover, $\pi_1$ acts on $\pi_k$ via the Whitehead product. In particular, if $k=1$, then $[f,g]=fgf^{-1}g^{-1}$. The Whitehead product vanishes for $H$-spaces \cite[Cor. 7.8]{Whitehead_elements}. Note that the Whitehead product is functorial, i.e., if $h:(X,x_0)\to (Y,y_0)$ is a continuous map, then $h_\ast [f,g]=[h_\ast f,h_\ast g]$.

\begin{defi}[Samelson product]
    Let $G$ be a topological group. Given $f:(S^k,1)\to (G,e)$ and $g:(S^l,1)\to (G,e)$, the \textit{Samelson product} of $[f],[g]$ is denoted by $\left\langle f,g\right\rangle\in \pi_{k+l}(G,e)$ and defined as the homotopy class of the map
    \begin{align*}
        S^{k+l}= S^k\wedge S^l \xrightarrow{f\wedge g} G\wedge G \xrightarrow{c} G,
    \end{align*}
    where $c:G\wedge G\to G$ is induced by the commutator map $G\times G \to G$ defined by $(g_1,g_2)\mapsto g_1g_2g_1^{-1}g_2^{-1}$.
\end{defi}

For a pointed space $(X,x_0)$, consider the based loop space $\Omega X$. Let $x_0$ also denote the constant loop at $x_0$. On $\Omega X$, loop concatenation defines an $H$-space structure with inverses and identity defined up to homotopy. Now for $f\in \pi_k(\Omega X,x_0),\ g\in \pi_l(\Omega X,x_0)$, we have the following diagram:
$$\begin{tikzcd}[row sep=large, column sep = huge]
    S^k\vee S^l \ar[d, hook]\ar[r, "H"] & (\Omega X)^I\ar[d, "ev_0"]\\
    S^k\times S^l \arrow[r, "{fg\bar{f}\bar{g}}"]\ar[ur,dashed, "\Tilde{H}"]
    & \Omega X 
\end{tikzcd}$$
where $\bar{f}, \bar{g}$ denote the pointwise inverses in $\Omega X$ and $H$ is induced from a homotopy between the map $fg\bar{f}\bar{g}$ and the constant map on $S^k\vee S^l$. As $S^k\vee S^l$ is a subcomplex of $S^k\times S^l$, it has a homotopy extension property, and hence we have a map 
$$\Tilde{H}: S^k\times S^l\to (\Omega X)^I=\text{Map}(I,\Omega X).$$
The map $\Tilde{S}:=\Tilde{H}(\cdot,\cdot)(1):S^k\times S^l\to \Omega X$ is constant on $S^k\vee S^l$ and thus induces a map $\langle f,g\rangle:S^{k+l}\cong S^k\wedge S^l\to \Omega X$. Therefore, $\langle f,g\rangle \in \pi_{k+l}(\Omega X, x_0)$ and this is called the \textit{Samelson product} of $f,g$. It is known that the following diagram commutes up to sign $(-1)^k$ \cite[Thm X.7.10]{Whitehead_elements}.
$$
\begin{tikzcd}[row sep=large, column sep = large]
    \pi_k(\Omega X)\times \pi_l(\Omega X) \arrow[r, "{\langle~,~\rangle}"] \arrow[d, "\ad\times \ad", "\cong"']
    &  \pi_{k+l}(\Omega X)\ar[d, "\ad", "\cong"']\\
    \pi_{k+1}(X)\times \pi_{l+1}(X) \arrow[r,"{[~,~]}"]
    & \pi_{k+l+1}(X)    
\end{tikzcd}
$$
where $\ad:\pi_n(\Omega X,x_0)\to \pi_{n+1}(X,x_0)$ is the adjunction isomorphism defined as follows: for $f:(I^n, \partial I^n)\to (\Om X, x_0)$ representing an element of $\pi_n(\Om X,x_0)$, consider the map $\ad(f):(I^{n+1}, \partial I^{n+1})\to (X, x_0)$ as
$$\ad(f)(t_0,t_1,\ldots,t_{n}):=f(t_1,\ldots,t_{n})(t_0)\text{ for $(t_0,t_1,\ldots, t_n)\in I^{n+1}$ },$$
which represents an element of $\pi_{n+1}(X,x_0)$.\medskip

For a fibration with a section, James defined a product relating the homotopy groups of the base and the fibre. Suppose $F\overset{i}{\hookrightarrow} E \xrightarrow{p} B$ is a fibration with a homotopy section $s:B\to E$, i.e., $p \comp s \simeq \mathrm{Id}_B$. Then the induced long exact sequence of homotopy groups splits, and {for each $j \ge 1$,} we get the following short exact sequence.
$$\begin{tikzcd}
    0\ar[r]&\pi_j(F)\ar[r,"i_*"] &\pi_j(E)\ar[r,"p_*"] &\pi_j(B)\ar[r] \ar[l,"s_*"', bend right] &0.
\end{tikzcd}$$
Now, for $\alpha\in \pi_k(B),\beta\in\pi_l(F)$, we have $p_*[s_*\alpha, i_*\beta]=0$ i.e., $[s_*\alpha, i_*\beta]\in \text{kernel}(p_*)=\text{image}(i_*)$ by exactness. The definition of brace product, due to James \cite{James_Products}, is as follows. 

\begin{defi}[Brace product] 
    Suppose $F\overset{i}{\hookrightarrow} E \xrightarrow{p} B$ is a fibration with a homotopy section $s:B\to E$. The \textit{brace product} $\{\cdot,\cdot\}_s:\pi_i(B)\times \pi_j(F)\to \pi_{i+j-1}(F)$ is defined such that
    $$i_*\{\alpha,\beta\}_s=[s_*\alpha, i_*\beta]\text{ or } \{\alpha,\beta\}_s=(i_*)^{-1}[s_*\alpha, i_*\beta].$$
\end{defi}
    
Since $i_*$ is injective and $[s_*\alpha, i_*\beta]\in \text{image}(i_*)$, the brace product is well-defined, and we have the following commutative diagram
$$
\begin{tikzcd}
    \pi_k(B)\times \pi_l(F) \arrow[r, "{\{~,~\}_s}"] \arrow[d, "s_*\times i_*"]
    &\pi_{k+l-1}(F) \arrow[d, "i_*"] \\
    \pi_k(E)\times \pi_{l}(E) \arrow[r,"{[~,~]}"]
    &\pi_{k+l-1}(E)
\end{tikzcd}$$

Note that the brace product is a degree $-1$ map which defines a $\pi_*(B)$-module structure on $\pi_*(F)$. In particular, for each $\alpha\in \pi_*(B)$, we have a map $D_\alpha:\pi_*(F)\to \pi_*(F)$ of degree $|\alpha|-1$ defined by $D_\alpha(\gamma)=\{\alpha,\gamma\}_s$ for $\gamma\in \pi_*(F)$. Since $\pi_*(F)$ is a graded quasi-Lie algebra with respect to the Whitehead product, in the first part of the following proposition, we observe that $D_\alpha$ is a derivation for all $\alpha\in \pi_*(B)$; whereas in the second part, we show that $D:\pi_*(B)\to \operatorname{Der}(\pi_*(F))$ sending $\alpha$ to $D_\alpha$ is a Lie algebra map, where $\operatorname{Der}(\pi_*(F))$ is the set of all derivations of the quasi-Lie algebra $\pi_*(F)$.

\begin{prop}
Let $\alpha\in \pi_j(B), \beta\in \pi_k(B), \gamma\in \pi_l(F)$, and $\delta\in \pi_m(F)$. Then,
\begin{enumerate}[\hspace{1em}(a)]
    \item $D_\beta$ is a derivation: $D_\beta([\gamma,\delta])=(-1)^{|D_\beta|}[D_\beta(\gamma),\delta]+(-1)^{|D_\beta|(|\gamma|-1)}[\gamma, D_\beta(\delta)]$, i.e.,  
    $$\left\{ \beta, [\gamma, \delta] \right\}_s= (-1)^{k-1}\left[\{\beta, \gamma\}_s, \delta\right]+ (-1)^{(l-1)(k-1)}\left[\gamma, \{\beta, \delta\}_s \right].$$
    \item \sloppy $D : \pi_*(B) \rightarrow \operatorname{Der}\left( \pi_*(F) \right)$ is a Lie algebra map: $D([\alpha,\beta])=(-1)^{|D_\alpha|}[D(\alpha),D(\beta)]$, i.e.,
    $$\left\{[\alpha, \beta], \gamma\right\}_s= (-1)^{j-1}\left( \left\{\alpha, \{\beta,\gamma\}_s\right\}_s- (-1)^{(j-1)(k-1)}\left\{\beta, \{\alpha, \gamma\}_s\right\}_s\right).$$
\end{enumerate}
\end{prop}
\begin{proof}
    (a) By the Jacobi identity of the Whitehead product in $\pi_*(E)$, we have \medskip

    \resizebox{0.97\textwidth}{!}{$\begin{aligned}
    &(-1)^{km}\left[[s_*(\beta), i_*(\gamma)], i_*(\delta)\right]+ (-1)^{kl}\left[[i_*(\gamma), i_*(\delta)], s_*(\beta)\right]+ (-1)^{lm}\left[[i_*(\delta), s_*(\beta)], i_*(\gamma)\right]=0\\
    \Rightarrow\ & (-1)^{km}\left[i_*\{\beta, \gamma\}_s, i_*(\delta)\right]+ (-1)^{kl}\left[i_*[\gamma, \delta], s_*(\beta)\right]+ (-1)^{lm}\left[(-1)^{km}i_*\{\beta, \delta\}_s, i_*(\gamma)\right]=0\\
    \Rightarrow\ & i_*\left((-1)^{km}\left[\{\beta, \gamma\}_s, \delta\right]+ (-1)^{kl+(l+m-1)k}\left\{ \beta, [\gamma, \delta] \right\}_s+ (-1)^{lm+km}\left[\{\beta, \delta\}_s, \gamma\right]\right)=0\\
    \Rightarrow\ & \left[\{\beta, \gamma\}_s, \delta\right]+ (-1)^{k}\left\{ \beta, [\gamma, \delta] \right\}_s+ (-1)^{lm}\left[\{\beta, \delta\}_s, \gamma\right]=0\\
    \Rightarrow\ & \left\{ \beta, [\gamma, \delta] \right\}_s= (-1)^{k-1}\left[\{\beta, \gamma\}_s, \delta\right]+ (-1)^{lm-k+1+l(k+m-1)}\left[\gamma, \{\beta, \delta\}_s\right]\\
    \Rightarrow\ & \left\{ \beta, [\gamma, \delta] \right\}_s= (-1)^{k-1}\left[\{\beta, \gamma\}_s, \delta\right]+ (-1)^{(l-1)(k-1)}\left[\gamma, \{\beta, \delta\}_s\right].
\end{aligned}$}\medskip

    (b) By the Jacobi identity of the Whitehead product in $\pi_*(E)$, we also have \medskip

\resizebox{0.97\textwidth}{!}{$\begin{aligned}
    &(-1)^{jl}\left[[s_*(\alpha), s_*(\beta)], i_*(\gamma)\right]+ (-1)^{kj}\left[[s_*(\beta),i_*(\gamma)], s_*(\alpha)\right]+ (-1)^{lk}\left[[i_*(\gamma), s_*(\alpha)], s_*(\beta)\right]=0\\
    \Rightarrow\ & (-1)^{jl}\left[s_*[\alpha, \beta], i_*(\gamma)\right]+ (-1)^{kj}\left[i_*\{\beta,\gamma\}_s, s_*(\alpha)\right]+ (-1)^{lk}\left[(-1)^{jl}i_*\{\alpha, \gamma\}_s, s_*(\beta)\right]=0\\
    \Rightarrow\ & (-1)^{jl}i_*\left\{[\alpha, \beta], \gamma\right\}_s+ (-1)^{kj+(k+l-1)j}i_*\left\{\alpha, \{\beta,\gamma\}_s\right\}_s+ (-1)^{lk+jl+(j+l-1)k}i_*\left\{\beta, \{\alpha, \gamma\}_s\right\}_s=0\\
    \Rightarrow\ & \left\{[\alpha, \beta], \gamma\right\}_s+ (-1)^{j}\left\{\alpha, \{\beta,\gamma\}_s\right\}_s+ (-1)^{(j-1)k}\left\{\beta, \{\alpha, \gamma\}_s\right\}_s=0\\
    \Rightarrow\ & \left\{[\alpha, \beta], \gamma\right\}_s= (-1)^{j-1}\left\{\alpha, \{\beta,\gamma\}_s\right\}_s- (-1)^{(j-1)k}\left\{\beta, \{\alpha, \gamma\}_s\right\}_s\\
    \Rightarrow\ & \left\{[\alpha, \beta], \gamma\right\}_s= (-1)^{j-1}\left( \left\{\alpha, \{\beta,\gamma\}_s\right\}_s- (-1)^{(j-1)(k-1)}\left\{\beta, \{\alpha, \gamma\}_s\right\}_s\right). 
\end{aligned}$} 

\end{proof}

The notion of Whitehead products and brace products can be generalized. We first need the following lemma.

\begin{lemma}\label{cofib_seq}
For based spaces $X,Y,$ and $W$, the cofibration sequence
$$ X\vee Y\xrightarrow{j} X\times  Y\xrightarrow{q}  X\wedge Y \xrightarrow{\rho} \Sigma(X\vee Y)\xrightarrow{\Sigma j} \Sigma(X\times Y)\rightarrow \cdots$$
induces an exact sequence of pointed sets.
\begin{align*}
    \{c \}\to [X\wedge Y,W]
    \xrightarrow{q^*} [ X\times  Y,W] \xrightarrow{j^*} [ X\vee  Y ,W],
\end{align*}
where $c$ is the constant map in $[ X\wedge  Y,W]$.
\end{lemma}
\begin{proof}
As the cofibration sequence induces a long exact sequence \cite[p. 134]{Whitehead_elements} 
\begin{align*}
    \cdots \rightarrow [\Sigma( X\times  Y),W] \xrightarrow{(\Sigma j)^*} [\Sigma( X&\vee  Y),W] \xrightarrow{\rho^*} 
        [ X\wedge  Y, W]\\
    &\xrightarrow{q^*} [ X\times  Y,W] \xrightarrow{j^*} [ X\vee  Y ,W],
\end{align*}
it is enough to show that $(\Sigma j)^*$ is onto. 
Using the $\Sigma$-$\Omega$ adjunction, it is then equivalent to show  that 
$$j^*:[ X\times  Y, \Omega W] \to [ X\vee  Y, \Omega W]$$ is onto.
Consider $f\in [X\vee Y, \Omega W]$ and define $\Tilde{f}(x, y) := f(x, y_0) \cdot f(x_0, y)$ for $(x, y) \in X \times Y$, where $x_0$ and $y_0$ are the base points in $X$ and $Y$, respectively, and `$\cdot$' represents the concatenation operation in $\Omega W$. Consequently, $\Tilde{f}\in [X \times Y, \Omega W]$, and the map $j^*(\Tilde{f}) = f$, implying that $j^*$ is onto.
\end{proof}

Note that if $X=Y$ and $W=\Omega X$, we have the commutator map $C\in [\Omega X\times \Omega X,\Omega X]$, and $j^*(C)$ is null-homotopic. {By Lemma~\ref{cofib_seq}, there exists a unique (up to homotopy) $S\in [\Omega X\wedge \Omega X,\Omega X]$ satisfying $q^*(S)=C$.} 
\begin{defi}[Samelson map]\label{defn:samelsonMap}
    A map $S\in [\Omega X\wedge \Omega X, \Omega X]$ such that $q^*(S)=C$ will be called the \textit{Samelson map}.
\end{defi}
The map $S$ induces the Samelson product, i.e., for $f\in \pi_k(\Omega X),g\in \pi_l(\Omega X)$, the Samelson product is given as $\langle f,g\rangle=S\comp (f\wedge g)$. The generalized Whitehead product is now defined as follows.

\begin{defi}[Generalized Whitehead product \cite{Arkowitz_gen_Whitehead}]
Let $\displaystyle \alpha \in [\Sigma A,X]$ and $\displaystyle \beta \in [\Sigma B,X]$ with  $\ad \alpha \in[A,\Omega X]$ and $\ad \beta \in[B,\Omega X]$ being the respective adjoints. The \textit{generalized Whitehead product} of $\alpha,\beta$ is defined as  $\ad^{-1}\left( S\comp (\ad \alpha \wedge \ad \beta) \right)\in [\Sigma (A\wedge B), X]$, where $S$ is the Samelson map. We shall denote this element by $[\alpha ,\beta ] \in [\Sigma (A\wedge B),X]$. 
\end{defi}

\begin{remark}\label{Gen_White_prop}
\begin{enumerate}[\quad (a)]
\item In \cite{Arkowitz_gen_Whitehead}, Arkowitz defined the generalized Whitehead product as follows. They identified a base point preserving map $\omega : \Sigma (A \wedge B) \rightarrow \Sigma A \vee \Sigma B$, called the \emph{universal Whitehead product map}, and defined the generalized Whitehead product of $f : \Sigma A \rightarrow X, g : \Sigma B \rightarrow X$ as the homotopy class of the composition
\[\begin{tikzcd}
\Sigma(A\wedge B) \arrow{r}{\omega} & \Sigma A \vee \Sigma B \arrow{r}{f\vee g} & X\vee X \arrow{r}{\nabla} & X 
\end{tikzcd},\]
where $\nabla$ is the fold map.
Using \cite[Thm 2.4]{Arkowitz_gen_Whitehead} and the $\Sigma$-$\Omega$ adjunction isomorphism, one can verify that both definitions agree. We will employ either definition, choosing the one that is more convenient based on the context. 
\item If we take $A=S^{k-1}, B=S^{l-1}$ in the above definition, then the generalized Whitehead product coincides with the classical Whitehead product. That is why the notations used for both products are not distinguished, but will be understood depending on the context.  
\item Generalized Whitehead product is functorial, i.e., $f_*([\alpha,\beta])=[f_*  \alpha,f_* \beta]$ for $f: X\to Y$.
\item If $X$ is an $H$-space, then $\Omega X$ is homotopy abelian following the Eckmann-Hilton argument. Thus, the Samelson map $S$ is null-homotopic and hence the generalized Whitehead product of $X$ vanishes identically.
\end{enumerate}
\end{remark}

\section{Generalized Brace Product} \label{sec:genBraceProductProperties}
Recall that given any fibration $F\overset{i}{\hookrightarrow} E \xrightarrow{p}B$, we have the following fibration sequence
\[\begin{tikzcd}[column sep=2em]
    \cdots \arrow{r}{} & \Omega^2 B \arrow{r}{\partial} & \Omega F \arrow{r}{\Omega i} & \Omega E \arrow{r}{\Omega p} & \Omega B \arrow{r}{\partial} & F \arrow{r}{i} & E \arrow{r}{p} & B.
\end{tikzcd}\]
This induces a long exact sequence of groups (and pointed sets) when one applies the functor $[X, -]$ for an arbitrary space $X$. The existence of a homotopy section $s: B \rightarrow  E$ then splits this long exact sequence into short exact sequences in groups
\begin{equation}\label{i_*IsMonic}
    \begin{tikzcd}
    0 \arrow{r}{} & {[\Sigma^i X, F]} \arrow{r}{i_*} & {[\Sigma^i X, E]} \arrow{r}{p_*} & {[\Sigma^i X, B]} \arrow{r}{} \ar[l,bend right,"s_*"']& 0
    \end{tikzcd}, \quad i \ge 1,
\end{equation}
where we have used the $\Sigma \text{-} \Omega$ adjunction as appropriate. Now, for $f : \Sigma X \to B, g : \Sigma Y \to F$, by  Remark~\ref{Gen_White_prop}(c) we have 
\[p_*[s \comp f, i \comp g] = [p \comp s \comp f, p \comp i \comp g] = [p \comp s \comp f, 0] = 0.\]

\begin{defi}[{Generalized brace product \cite[Def 3.3]{Yoon-DecomposabilityOfEvaluationFibration}}] \label{defn:genBraceProduct}
    Consider a fibration $F\overset{i}{\hookrightarrow} E \xrightarrow{p}B$ admitting a homotopy section $s$. Then for any $f:\Sigma X\to B$ and $g:\Sigma Y\to F$, the element $\left\{ f, g \right\}_s \in \left[ \Sigma\left( X\wedge Y \right), F \right]$ is defined by the equation 
    \begin{equation} \label{eq:genBraceDefinition}
        \left\{ f, g \right\}_s = i_*^{-1}[s_*f, i_* g].
    \end{equation}
    This defines a map 
    $$
    \{,\}_s:[\Sigma X,B]\times [\Sigma Y,F] \to [\Sigma(X\wedge Y),F],
    $$
    which is called the \textit{generalized brace product}.
\end{defi} 

Note that if $\Sigma X,\Sigma Y$ both are spheres, then it follows from the definition itself and Remark~\ref{Gen_White_prop}(b) that the generalized brace product and brace product coincide. Consequently, we reuse the same symbol $\{,\}_s$ to denote the generalized brace product as well. Henceforth, if it is necessary to distinguish between the two, we shall refer to the case of spheres as the James brace product.

Analogous to the Samelson product, the generalized brace product (and hence the standard one as well) comes from a space-level map. Kallel and Sjerve \cite[Lemma 2.9]{Kallel-Sjerve} first observed this for the brace product in homotopy groups, although their proof was inadequate. Here we prove this fact in complete detail.

\begin{prop}\label{map_inducing_brace}
    For a fibration $F\overset{i}{\hookrightarrow} E\overset{p}{\to} B$ with a homotopy section $s$, there is a map $J_s:\Omega B\wedge \Omega F\to \Omega F$, which induces the generalized brace product. Moreover, this map is unique up to homotopy.
\end{prop}
\begin{proof}
Recall from Definition~\ref{defn:samelsonMap}, we have the commutator map $C: \Omega E \times \Omega E \rightarrow \Omega E$, and the Samelson map $S: \Omega E \wedge \Omega E \rightarrow \Omega E$, such that $S\comp q \simeq C$, where $q: \Omega E \times \Omega E \rightarrow \Omega E \wedge \Omega E$ is the quotient map.
Now, from the following commutative diagram
\begin{equation*}
    \begin{tikzcd}
        \Omega B\vee \Omega F\ar[r,hook,"j'"] \ar[d,"\Omega s\vee \Omega i"] &\Omega B\times \Omega F \ar[r, "q'"] \ar[d, "\Omega s\times \Omega i"]  &\Omega B\wedge \Omega F\ar[d, "\Omega s\wedge \Omega i"]\\
        \Omega E\vee \Omega E\ar[r,hook,"j"]  &\Omega E\times \Omega E \ar[r, "q"]  &\Omega E\wedge \Omega E
    \end{tikzcd}
\end{equation*}
we have
\begin{align*}
    \Omega p\comp S\comp (\Omega s\wedge \Omega i)\comp q' 
    &= \Omega p\comp S\comp q\comp (\Omega s\times \Omega i)\\
    &\simeq \Omega p\comp C\comp (\Omega s\times \Omega i)\\
    &= \Omega p\comp (\Omega s\cdot \Omega i \cdot \overline{\Omega s}\cdot \overline{\Omega i})\\
    &= \Omega (p\comp s)\cdot \Omega (p\comp i) \cdot \overline{\Omega (p\comp s)}\cdot \overline{\Omega (p\comp i)}\\
    &\simeq \Omega (\textrm{Id}_B)\cdot \Omega c_b \cdot \overline{\Omega (\textrm{Id}_B)}\cdot \overline{\Omega c_b}\\ 
    &\qquad \qquad (\text{as $p\comp s\simeq \textrm{Id}_B$, and $p\comp i=$ constant map, say $c_b$})\\
    &\simeq \Omega c_b
\end{align*}
Given $X = \Omega B$, $Y = \Omega F$, and $W = \Omega B$, Lemma~\ref{cofib_seq} implies that $(q')^*$ is injective. Consequently, the preceding computation yields $\Omega p \comp S \comp (\Omega s \wedge \Omega i) \simeq \Omega c_b$.  
Then, applying the functor $[\Omega B\wedge\Omega F, -]$ to  the following sequence of fibration
$$\Omega^2E\xrightarrow{\Omega^2p} \Omega^2B\xrightarrow{\Omega \partial} \Omega F\xrightarrow{\Omega i}\Omega E\xrightarrow{\Omega p}\Omega B,$$
and noting that $\Omega^2 p$ has right inverse $\Omega^2s$, 
we obtain the following exact sequence 
$$\left\{ \ast \right\} \to [\Omega B\wedge \Omega F,\Omega F]\xrightarrow{(\Omega i)_*}[\Omega B\wedge \Omega F,\Omega E] \xrightarrow{(\Omega p)_*} [\Omega B\wedge \Omega F, \Omega B].$$
As $S\comp (\Omega s\wedge \Omega i)\in \text{Ker}(\Omega p)_*$, we have a unique (up to homotopy) map $J_s:\Omega B\wedge \Omega F\to \Omega F$ such that $(\Omega i)_*(J_s)=S\comp (\Omega s\wedge \Omega i)$.  
Now for $f:\Sigma X\to B$ and $g:\Sigma Y\to F$, we have the following

\begin{align*}
    i_*(\{f,g\}_s)&=[s\comp f,i\comp g]\\
    &=\ad^{-1} \left(S\comp( \ad (s \comp f)\wedge \ad (i \comp g)) \right) \\
    &=\ad^{-1} \left( S\comp ((\Omega s\comp \ad f)\wedge (\Omega i\comp \ad g)) \right)\\
    &=\ad^{-1} \left( S\comp (\Omega s\wedge \Omega i)\comp (\ad f\wedge \ad g) \right)\\
    &=\ad^{-1} \left( (\Omega i)_*\left(J_s\comp (\ad f\wedge \ad g)\right) \right)\text{ (by the construction of $J_s$)}\\
    &= i_*\left(\ad^{-1} \left(J_s\comp (\ad f\wedge \ad g)\right)\right)
\end{align*}
    Since $i_*$ is monic (see (\ref{i_*IsMonic})), we get $\{f,g\}_s=\ad^{-1}\left(J_s\comp (\ad {f}\wedge\ad {g})\right)$. 
    Thus, $J_s$ induces the generalized brace product $\{,\}_s$.  
    
    Observe that by taking $f_0 = \ad^{-1}(\mathrm{Id}_{\Omega B}) : \Sigma \Omega B \rightarrow B$ and $g_0 = \ad^{-1}(\mathrm{Id}_{\Omega F}) : \Sigma \Omega F \rightarrow F$, we have 
    \[\left\{ f_0, g_0 \right\}_s = \ad^{-1} \left( J_s \comp \left( \ad f_0 \wedge \ad g_0 \right) \right) = \ad^{-1}\left( J_s \comp \left( \mathrm{Id}_{\Omega B} \wedge \mathrm{Id}_{\Omega F} \right) \right) = \ad^{-1} J_s.\]
    Hence, $J_s = \ad \left\{ f_0, g_0 \right\}_s$ is defined, uniquely up to homotopy.
\end{proof}
\begin{remark}\label{remark:UniversalBrace}
    In the above Proposition~\ref{map_inducing_brace}, we have seen that the generalized brace product $\{f,g\}_s$ of any two maps $f:\Sigma X\to B$ and $g:\Sigma Y\to F$ factors through the brace product $\{f_0,g_0\}_s$ of $f_0:\Sigma \Omega B\to B$ and $g_0:\Sigma \Omega F\to F$, in the sense that
    $$\ad \{f,g\}_s=\ad \{f_0, g_0\}_s\comp (\ad f\wedge \ad g).$$
    Thus, we can say the $J_s$ map or the brace product $\{f_0,g_0\}_s$ is universal for the generalized brace product.
\end{remark}

\subsection{Independence of base point and homotopy of sections}
In this section, we briefly discuss the base point independence of the generalized brace product. Given a well-pointed space $(A,*)$ (i.e., $* \hookrightarrow A$ is a cofibration), and a path $\gamma : [0, 1] \rightarrow X$ joining $x_0 = \gamma(0)$ to $x_1 = \gamma(1)$, we can define a canonical morphism 
\[\mathsf{h}_{[\gamma]} : [(A, *), (X,x_0)] \rightarrow [(A, *), (X,x_1)]\]
using the homotopy extension property in the following way. Given any map $f : (A, *) \rightarrow (X, x_0)$,  we have the diagram 
\begin{equation} \label{eqn:basepointChangeMap}
\begin{tikzcd}
    0 \arrow{r}{f} \arrow[hookrightarrow]{d} & X^A \arrow{d}{ev_*} \\
    {[0,1]} \arrow{r}[swap]{\gamma} \arrow[dotted]{ru}[swap]{\tilde{f}} & X
\end{tikzcd}
\end{equation}
which admits a lift $\tilde{f}$ since $(A, *)$ is well-pointed. This lift is unique up to homotopy for any homotopy of $\gamma$ relative to $\left\{ 0,1 \right\}$ and any base point preserving homotopy of $f$. Consequently, we may define $\mathsf{h}_{[\gamma]}([f]) = [\tilde{f}(1)]$. It is easy to see that $\mathsf{h}_{[\gamma]}$ is a bijection, with the inverse given by $\mathsf{h}_{[\bar{\gamma}]}$. Note that for any $s : X \rightarrow  Y$ we have
\begin{equation}\label{eqn:basePointChangeMapPostcompose}
\mathsf{h}_{[s \comp \gamma]} \comp s_* = s_* \comp h_{[\gamma]}.
\end{equation}
Furthermore, given $f_i : (A_i, *) \rightarrow (X, x_0), i=1,2$, we have a lift of $f_1 \vee  f_2$, which is homotopic to the wedge of their individual lifts. Consequently, if $A$ is a suspension, then $\mathsf{h}_{[\gamma]}$ is a group isomorphism. We have the following lemma.
\begin{lemma}\label{lemma:genWhiteheadBasepoint}
The generalized Whitehead product is independent of base points, i.e., for any well-pointed spaces $(\Sigma A, *)$ and $(\Sigma B, *)$, we have a commutative diagram 
\[\begin{tikzcd}
    {[(\Sigma A, *), (X, x_0)] \times [(\Sigma B, *), (X, x_0)]} \arrow{r}{[~,~]} \arrow{d}[swap]{\mathsf{h}_{[\gamma]}\times \mathsf{h}_{[\gamma]}} & {[(\Sigma (A \wedge B),*), (X,x_0)]} \arrow{d}{\mathsf{h}_{[\gamma]}} \\
    {[(\Sigma A, *), (X,x_1)] \times [(\Sigma B, *), (X, x_1)]} \arrow{r}[swap]{[~,~]} & {[(\Sigma (A \wedge B), *), (X, x_1)]}
\end{tikzcd}\]
\end{lemma}
\begin{proof}
Recall, we have the base point preserving map $\omega : \Sigma (A \wedge B) \rightarrow \Sigma A \vee \Sigma B$, so that the (generalized) Whitehead product of $f : (\Sigma A, *) \rightarrow (X, x_0)$ and $g : (\Sigma B, *) \rightarrow (X, x_0)$ is represented by the composition 
\[\begin{tikzcd}
\Sigma (A \wedge B) \arrow{r}{\omega} & \Sigma A \vee \Sigma B \arrow{r}{f \vee g} & X \vee X \arrow{r}{\nabla} & X 
\end{tikzcd}.\]
Since in Diagram \eqref{eqn:basepointChangeMap} the lift of $\nabla \comp (f \vee g)$ is homotopic to $\nabla \comp (\tilde{f} \vee \tilde{g})$, the claim follows by precomposing with $\omega$.
\end{proof}

Let us now consider a fibration $p: E \rightarrow B$, with section $s: B \rightarrow E$. Given two base points $b_0, b_1 \in B$, consider a path $\gamma : [0,1] \rightarrow B$ joining $b_0 = \gamma(0)$ to $b_1= \gamma(1)$. Denote the fibres $F_i = p^{-1}(b_i), i = 0, 1$, and fix the base points $e_i = s(b_i) \in F_i \subset E$. Then, $p : (E, e_i) \rightarrow (B, b_i)$ is base point preserving for $i = 0, 1$. Denote the inclusion maps $\iota_j : (F_j, e_j) \hookrightarrow (E, e_j), j= 0, 1$, and set  $\tilde{\gamma} = s \comp \gamma$. We have the following diagram 
\begin{equation} \label{eqn:basepointChangeWithSection}
\begin{tikzcd}[column sep=2cm]
    F_0 \vee [0, 1] = F_0 \times \left\{ 0 \right\} \cup \left\{ e_0 \right\} \times [0,1] \arrow{r}{\iota_0 \vee \tilde{\gamma}} \arrow[hookrightarrow]{d}[swap]{} & E \arrow{d}{p} \\
    F_0 \times [0, 1] \arrow{r}[swap]{\gamma \comp \text{pr}_2} \arrow[dotted]{ru}[swap]{\Gamma} & B
\end{tikzcd}
\end{equation}
Since $(F_0, e_0)$ is well-pointed, $F_0 \vee [0,1] \hookrightarrow F_0 \times [0,1]$ is a cofibration (see \cite[Thm 6]{Strom-NoteOnCofibration}), and also a homotopy equivalence. Thus, we have a lift, say, $\Gamma$.

Clearly, $p \comp \Gamma(\text{-}, t) = \gamma(t)$, and $\Gamma(e_0, t) = \tilde{\gamma}(t)$. In particular, we have a map $\mathsf{L}_\gamma : (F_0, e_0)\rightarrow (F_1, e_1)$ defined as $\Gamma(\text{-}, 1)$. It follows that if $\gamma \simeq \gamma^\prime (\text{rel} \left\{ 0,1 \right\})$, then $\mathsf{L}_\gamma \simeq \mathsf{L}_{\gamma^\prime}$, and furthermore $\mathsf{L}_{\gamma \cdot \eta} \simeq \mathsf{L}_\eta \comp \mathsf{L}_\gamma$ whenever $\gamma, \eta$ are two paths in $B$ that can be concatenated \cite[p. 405]{Hatcher}. Consequently, we get a homotopy equivalence $\mathsf{L}_{\gamma} : (F_0, e_0) \rightarrow (F_1, e_1)$. With the above notations, we have the following result.

\begin{prop}\label{prop:genBraceBasepoint}
The generalized brace product is independent of choice base points in the following sense: for any $(\Sigma X, *)$ and $(\Sigma Y, *)$, the following diagram commutes 
\[\begin{tikzcd}
    \left[ (\Sigma X,*), (B,b_0) \right] \times \left[ (\Sigma Y,*), \left( F_0, e_0 \right) \right] \arrow{r}{\left\{ , \right\}_s} \arrow{d}[swap]{\mathsf{h}_{[\gamma]}\times (\mathsf{L}_{\gamma})_*} & \left[ \left( \Sigma (X \wedge Y), * \right), \left( F_0, e_0 \right) \right] \arrow{d}{(\mathsf{L}_{\gamma})_*} \\
    \left[ (\Sigma X, *), (B, b_1) \right] \times \left[ (\Sigma Y, *), \left( F_1, e_1) \right) \right] \arrow{r}[swap]{\left\{ , \right\}^\prime_s} & \left[ \left( \Sigma(X\wedge Y), * \right), \left( F_1, e_1 \right) \right]
\end{tikzcd}\]
\end{prop}
\begin{proof}
Observe that, for any $(Z, *)$, we have a commutative diagram
\begin{equation}\label{prop:genBraceBasepoint:1}
    \begin{tikzcd}
        {[(Z,*), (F_0,e_0)]} \arrow{r}{(\iota_0)_*} \arrow{d}[swap]{(\mathsf{L}_{\gamma})_*} & {[(Z,*), (E, e_0)]} \arrow{d}{\mathsf{h}_{[s\comp \gamma]}} \\
        {[(Z,*), (F_1,e_1)]} \arrow{r}[swap]{(\iota_1)_*} & {[(Z,*), (E, e_1)]}    
    \end{tikzcd}
\end{equation}
since, $\Gamma$, as in Diagram (\ref{eqn:basepointChangeWithSection}), is a lift of $\iota_0 \vee (s \comp \gamma)$ by construction. Also, from Lemma~\ref{lemma:genWhiteheadBasepoint}, we have the following commutative diagram
\begin{equation}\label{prop:genBraceBasepoint:2}
    \begin{tikzcd}
        \left[(\Sigma X, *), (E, e_0) \right] \times \left[ (\Sigma Y, *), (E, e_0) \right] \arrow{r}{[~,~]} \arrow{d}[swap]{\mathsf{h}_{[s\comp \gamma]} \times \mathsf{h}_{[s\comp \gamma]}} & \left[ (\Sigma(X \wedge Y), *), (E, e_0) \right] \arrow{d}{\mathsf{h}_{[s\comp \gamma]}} \\
        \left[ (\Sigma X, *), (E, e_1) \right] \times \left[ (\Sigma Y, *), (E, e_1) \right] \arrow{r}[swap]{[~,~]} & \left[ (\Sigma(X \wedge Y), *), (E, e_1) \right]
    \end{tikzcd}
\end{equation}
Hence, given $\alpha \in [(\Sigma X, *), (B,b)]$ and $\beta \in [(\Sigma Y, *), (E_b, s(b))]$ we have
\begin{alignat*}{3}
    (\mathsf{L}_{\gamma})_* \left\{ \alpha, \beta \right\}_s &= (\mathsf{L}_{\gamma})_*  (\iota_0)_*^{-1}[s_* \alpha, (\iota_0)_* \beta] \\
    &= (\iota_1)_*^{-1} \mathsf{h}_{[s\comp \gamma]}[s_* \alpha, (\iota_0)_* \beta] \\
    &= (\iota_1)_*^{-1} \bigl[ \mathsf{h}_{[s\comp \gamma]} s_* \alpha, \mathsf{h}_{[s\comp \gamma]} (\iota_0)_*\beta\bigr] &&\qquad (\text{by Diagram \eqref{prop:genBraceBasepoint:2}})\\
    &= (\iota_1)_*^{-1} \bigl[ s_* \left( \mathsf{h}_{[\gamma]} \alpha \right), (\iota_1)_* (\mathsf{L}_{\gamma})_* \beta \bigr] &&\qquad (\text{by Diagram \eqref{prop:genBraceBasepoint:1} and \eqref{eqn:basePointChangeMapPostcompose}})\\
    &= \left\{ \mathsf{h}_{[\gamma]} \alpha, (\mathsf{L}_{\gamma})_* \beta \right\}^\prime_s.
\end{alignat*}
This proves the claim.
\end{proof}

For the rest of the article, we shall tacitly assume that the base points of the spaces are understood from the context. In a similar vein, let us also show that the brace product is independent of the choice of any vertical homotopy of the section.

\begin{defi}
Given two sections $s_0, s_1: B \rightarrow E$ of the fibration $p: E \rightarrow B$, we say $s_0$ and $s_1$ are \emph{vertically homotopic} if there exists a homotopy $s_t: B \rightarrow E$ of sections.
\end{defi}

\begin{prop}\label{prop:braceProductHomotopySection}
Let $s_t: B \rightarrow E$ be a vertical homotopy of sections. Denote, $F = p^{-1}(b_0)$ and consider the two base points $e_i = s_i(b_0) \in F$ for $i = 0, 1$. Then, for any $(\Sigma X, *)$ and $(\Sigma Y, *)$, the following diagram commutes
\[\begin{tikzcd}
    \left[ (\Sigma X, *), (B, b_0) \right] \times \left[ (\Sigma Y, *), (F, e_0) \right] \arrow{r}{\left\{ , \right\}_{s_0}} \arrow{d}[swap]{\mathrm{Id} \times \mathsf{h}_{[\gamma]}} & \left[ \Sigma(X \wedge Y), (F, e_0) \right] \arrow{d}{\mathsf{h}_{[\gamma]}} \\
    \left[ (\Sigma X, *), (B, b_0) \right] \times \left[ (\Sigma Y, *), (F, e_1) \right] \arrow{r}[swap]{\left\{ , \right\}_{s_1}} & \left[ \Sigma(X \wedge Y), (F, e_1) \right]
\end{tikzcd},\]
where $\gamma(t) = s_t(b_0)$ is a path in $F$ joining $e_0$ to $e_1$.
\end{prop}
\begin{proof}
Consider the inclusion map $\iota: F \hookrightarrow E$. Then, we have a commutative diagram 
\begin{equation}\label{prop:braceProductHomotopySection:1}
    \begin{tikzcd}
        \left[ (\Sigma Y, *), (F, e_0) \right] \arrow{r}{\iota_*} \arrow{d}[swap]{\mathsf{h}_{[\gamma]}} & \left[ (\Sigma Y, *), (E, s_0(b_0)) \right] \arrow{d}{\mathsf{h}_{[\iota \comp \gamma]}} \\
        \left[ (\Sigma Y, *), (F, e_1) \right] \arrow{r}[swap]{\iota_*} & \left[ (\Sigma Y, *), (E, s_1(b_0)) \right]
    \end{tikzcd}
\end{equation}
Also, for any $\alpha : (Z, *) \rightarrow (B, b_0)$, note that $(z,t) \mapsto s_t \comp \alpha(z)$ gives rise to a lift in Diagram~\eqref{eqn:basepointChangeMap}, and so we have $\mathsf{h}_{\iota \comp \gamma}\left( [s_0 \comp \alpha] \right) = [s_1 \comp \alpha]$. In particular, the diagram
\begin{equation}\label{prop:braceProductHomotopySection:2}
    \begin{tikzcd}
        \left[ (\Sigma X, *), (B, b_0) \right] \arrow{r}{(s_0)_*} \arrow{d}[swap]{\mathrm{Id}} & \left[ (\Sigma X, *), (E, e_0) \right] \arrow{d}{\mathsf{h}_{[\iota \comp \gamma]}} \\
        \left[ (\Sigma X, *), (B, b_0) \right] \arrow{r}[swap]{(s_1)_*} & \left[ (\Sigma X, *), (E, e_1) \right]
    \end{tikzcd}
\end{equation}
commutes as well. Hence, for $\alpha \in [(\Sigma X, *), (B,b_0)]$ and $\beta \in [(\Sigma Y, *), (F, e_0)]$, we have
\begin{alignat*}{3}
    \mathsf{h}_{[\gamma]} \left\{ \alpha, \beta \right\}_{s_0} &= \mathsf{h}_{[\gamma]} (\iota_*)^{-1}[(s_0)_* \alpha, \iota_* \beta] \\
    &= (\iota_*)^{-1} \mathsf{h}_{[\iota \comp \gamma]} [(s_0)_* \alpha, \iota_* \beta] \\
    &= (\iota_*)^{-1} [\mathsf{h}_{[\iota \comp \gamma]} (s_0)_* \alpha, \mathsf{h}_{[\iota \comp \gamma]} \iota_* \beta] &&\qquad (\text{by Diagram~\ref{lemma:genWhiteheadBasepoint}})\\
    &= (\iota_*)^{-1} [(s_1)_* \alpha, \iota_* \mathsf{h}_{[\gamma]}\beta] &&\qquad (\text{by Diagram~\eqref{prop:braceProductHomotopySection:1} and~\eqref{prop:braceProductHomotopySection:2}})\\
    &= \left\{ \alpha, \mathsf{h}_{[\gamma]}\beta \right\}_{s_1}
\end{alignat*}
This completes the proof.
\end{proof}

\subsection{Brace product of pullback and product fibrations}
In this section, we analyze how the generalized brace product behaves under the usual constructions on a fibration, such as pullback and product. The following lemma relates the generalized brace product of a fibration and its pullback fibration by a map.
\begin{lemma}\label{pullback_brace}
Suppose $F\overset{i}{\hookrightarrow} E\overset{p}{\to} B$ is a fibration with a (homotopy) section $s: B\to E$ and $f: \Tilde{B}\to B$ is a continuous map. Then 
\begin{align*}
        \{\alpha,\beta\}_{f^*s}=\{f_*(\alpha),\beta\}_s\text{ for } \alpha\in  [\Sigma X,  \Tilde{B}] , \beta \in [\Sigma Y, F],
\end{align*}
where $\{,\}_{f^*s}$ and $\{,\}_s$ denote the generalized brace product of the pullback fibration $f^*E$ and $E$ with respect to the pullback section $f^*s$ and the section $s$ respectively.
\end{lemma}

The proof is immediate from the following more general case.
\begin{lemma}\label{lemma:mapOfFibrations}
    Suppose there are two fibrations with sections and a diagram
    \[\begin{tikzcd}
        F_1 \arrow{r}{i_1} \arrow{d}[swap]{\varphi} & E_1 \arrow{d}{\Phi} \arrow{r}[swap]{p_1} & B_1 \arrow[bend right=30]{l}[swap]{s_1} \arrow{d}{\phi} \\
        F_2 \arrow{r}[swap]{i_2} & E_2 \arrow{r}{p_2} & B_2 \arrow[bend left=30]{l}{s_2}
    \end{tikzcd}\]
    such that $\Phi \comp \left( s_1 \vee i_1 \right) \simeq \left( s_2 \vee i_2 \right) \comp \left( \phi \vee \varphi \right) = \left( s_2 \comp \phi \right) \vee \left( i_2 \comp \varphi \right)$. Then, for any $\alpha \in \left[ \Sigma X, B_1 \right], \beta \in \left[ \Sigma Y_1, F_1  \right]$ we have 
    \[\varphi_* \left\{ \alpha, \beta \right\}_{s_1} = \left\{ \phi_* \alpha, \varphi_* \beta \right\}_{s_2}.\]
\end{lemma}
\begin{proof}
    We have
    \begin{align*}
        \left( i_2 \right)_* \varphi_* \left\{ \alpha, \beta \right\}_{s_1}
        &= \Phi_* \left( i_1 \right)_* \left\{ \alpha, \beta \right\}_{s_1} \\
        &= \Phi_* \left[ \left( s_1 \right)_* \alpha, \left( i_1 \right)_* \beta \right] \\
        &= \left[ \Phi_* \left( s_1 \right)_* \alpha, \Phi_* \left( i_1 \right)_* \beta \right] \\
        &= \left[ \left( s_2 \right)_* \phi_* \alpha, \left( i_2 \right)_* \varphi_* \beta \right] \\
        &= \left( i_2 \right)_* \left\{ \phi_* \alpha, \varphi_* \beta \right\}_{s_2}
    \end{align*}
    As $\left( i_2 \right)_*$ is monic, the proof follows.
\end{proof}

Let us now observe how the generalized brace product of the product of two fibrations depends on those of the individual fibrations.
\begin{lemma}
Suppose $F_i\overset{j_i}{\hookrightarrow} E_i\xrightarrow{p_i} B$ is a fibration with homotopy section $s_i$ for $i=1,2$. Then $F:=F_1\times F_2 \lhook\joinrel\xrightarrow{(j_1,j_2)} E:=E_1\times E_2\xrightarrow{(p_1,p_2)} B\times B$ is the product fibration with homotopy section $s:=(s_1,s_2)$. If $\Delta: B\to B\times B$ denotes the diagonal map, then the generalized brace product in the pullback bundle $\Delta^*(E)$ with respect to the pullback homotopy section $\Delta^*(s)$ is given as follows.
\begin{align*}
    \{\alpha,\beta\}_{\Delta^*(s)}=\left(\{\alpha, \beta_1\}_{s_1}, \{\alpha, \beta_2\}_{s_2}\right) \text{ for } \alpha\in  [\Sigma X,  B] , \beta=(\beta_1,\beta_2) \in [\Sigma Y, F].
\end{align*}
\end{lemma}
\begin{proof}
By Lemma~\ref{pullback_brace}, we have $\{\alpha,\beta\}_{\Delta^*(s)}=\{\Delta_*(\alpha),\beta\}_s=\{(\alpha,\alpha),(\beta_1,\beta_2)\}_s$. Now 
\begin{align*}
    (j_1,j_2)_*\left( \{(\alpha,\alpha),(\beta_1,\beta_2)\}_s \right)
    &= \left[ s_*(\alpha,\alpha), (j_1,j_2)_*(\beta_1,\beta_2)\right]\\
    &= \left[ ((s_1)_*(\alpha), (s_2)_*(\alpha)), ((j_1)_*(\beta_1), (j_2)_*(\beta_2))\right]\\
    &= \left( [(s_1)_*(\alpha), (j_1)_*(\beta_1)], [(s_2)_*(\alpha), (j_2)_*(\beta_2)]\right)\\
    &= \left( (j_1)_*\{\alpha,\beta_1\}_{s_1}, (j_2)_*\{\alpha,\beta_2\}_{s_2} \right)\\
    &= (j_1,j_2)_*\left( \{\alpha,\beta_1\}_{s_1}, \{\alpha,\beta_2\}_{s_2} \right)
\end{align*}
As $(j_1,j_2)_*$ is injective, the above equality implies 
\[\{(\alpha,\alpha),(\beta_1,\beta_2)\}_s = \left(\{\alpha, \beta_1\}_{s_1}, \{\alpha, \beta_2\}_{s_2}\right).\qedhere\]
\end{proof}

\subsection{Brace product in free loop space fibration}
Suppose $Z$ is a simple space, i.e., $\pi_1(Z)$ is abelian and its action on the higher homotopy groups of $Z$ is trivial. Let $\mathcal{L}^m_0 Z$ be the component of the constant map in $\mathcal{L}^mZ:=\textup{Map}(S^m, Z)$. Let $\Omega^m Z$ be the corresponding component in the space of based maps that contains the constant based map. Then the evaluation map $ev:\mathcal{L}_0^m Z\to Z$ at the base point of $S^m$ defines a (Hurewicz) fibration $\Omega^m Z\overset{i}{\hookrightarrow}  \mathcal{L}_0^m Z \xrightarrow{ev} Z$ with the canonical section $s$ of constant maps. We will refer to this as the $m^{\text{th}}$ free loop space fibration of $Z$.
Hansen \cite[Thm 2.1]{Hansen} proved a natural relation between the Whitehead product of $Z$ and the brace product in this fibration. Later, Yoon \cite[Thm 3.4]{Yoon-DecomposabilityOfEvaluationFibration} established a similar relation between the generalized Whitehead product and the generalized brace product in any general evaluation fibration. In the special case of the $m^{\text{th}}$ free loop space fibration, we reestablish the relation in the following explicit form.

{\begin{thm}[{\cite[Thm 3.4]{Yoon-DecomposabilityOfEvaluationFibration}}]\label{brace_whitehead_in_free_loop}
    Suppose $f\in [\Sigma X, Z]$, $g\in [\Sigma Y,\Omega^m Z]$. Then, in the $m^\text{th}$ free loop space fibration as above, we have
    $$\{f,g\}_s =\ad^m [f,\ad^{-m}(g)],$$
    where
    \[\begin{tikzcd}
        \left[ \Sigma^k U, V \right] \arrow[yshift=3pt]{r}{\ad^k} & \left[ U, \Omega^k V \right] \arrow[yshift=-3pt]{l}{\ad^{-k}}
    \end{tikzcd}\]
    denote the $\Sigma$-$\Omega$ adjoint isomorphisms repeated $k$-times for arbitrary spaces $U, V$.
\end{thm}
\begin{proof}
    The element $ i_* \left\{ f, g \right\}_s$ is represented by the map
    \[\ad^{-1}\left( S \comp \left( \Omega s \wedge \Omega i \right) \comp \left( \ad f \wedge \ad g \right) \right) = \ad^{-1} \left( S \comp \left( \left( \Omega s \comp \ad f \right) \wedge \left( \Omega i \comp \ad g \right) \right) \right).\]
    Here $S : \Omega\left( \mathcal{L}^m_0 Z\right) \wedge \Omega\left(\mathcal{L}^m_0 Z \right) \rightarrow \Omega\mathcal{L}^m_0 Z$ is a Samelson map satisfying $S \comp q \simeq C$, $C$ is the commutator and $q$ is the quotient map to the smash product.
    Similarly, $\ad^m[f,\ad^{-m}(g)]$ is represented by the map
    $$\ad^m\left(\ad^{-1}\left(S'\comp (\ad(f)\wedge \ad(\ad^{-m}(g)))\right)\right),$$
    where $S'$ is a Samelson map induced from the commutator map $C'$ on $\Omega Z\times \Omega Z$.
    Then, for given $(x \wedge y )\in X \wedge Y$, this map gives a loop in $\mathcal{L}^m_0 Z$ via
    \begin{align*}
        & t \mapsto \left( S \comp \left( \left( \Omega s \comp \ad f \right) \wedge \left( \Omega i \comp \ad g \right) \right) \left( x \wedge y \right) \right) (t) \\
        \simeq \ & t \mapsto \left( C \left( \Omega s \comp \ad f (x), \Omega i \comp \ad g(y) \right) \right)(t) \\
        =\ & \left( t \mapsto \left( \Omega s \comp \ad f(x) \right)(t) \right) \cdot \left( t \mapsto \left( \Omega i \comp \ad g(y) \right)(t) \right) \\
        &\qquad \cdot \left( t \mapsto \left( \Omega s \comp \ad f(x) \right)(\bar{t}) \right) \cdot \left( t \mapsto \left( \Omega i \comp \ad g(y) \right)(\bar{t}) \right) \\
        =\ & \left( t \mapsto s \comp f(x \wedge t) \right) \cdot \left( t \mapsto i \comp g(y \wedge t)\right) \cdot \left( t \mapsto s \comp f\left( x \wedge \bar{t} \right) \right) \cdot \left( t \mapsto i\comp g(y \wedge \bar{t})\right),
    \end{align*}
    where $\bar{t}$ denotes the reversed parameter in the suspension. Next, evaluating at $\theta \in S^m$, and noting that $s(z)$ is the constant map $\theta \mapsto z$ for any $z \in Z$, we have a loop in $Z$ given by
    \begin{align*}
        &\left( t \mapsto f(x \wedge t) \right) \cdot \left( t \mapsto  g(y \wedge t)(\theta) \right) \cdot \left( t \mapsto f(x \wedge \bar{t}) \right) \cdot \left( t \mapsto g(y \wedge \bar{t})(\theta) \right) \\
        =\ & \left( t \mapsto f(x \wedge t) \right) \cdot \left( t \mapsto  \ad^{-m} g(y \wedge t \wedge \theta) \right) \\
        &\qquad \cdot \left( t \mapsto f(x \wedge \bar{t}) \right) \cdot \left( t \mapsto \ad^{-m} g(y \wedge \bar{t} \wedge \theta) \right) \\
        =\ & \left( t \mapsto f(x \wedge t) \right) \cdot \left( t \mapsto  \ad^{-m} g(y \wedge \theta \wedge t) \right) \\
        &\qquad \cdot \left( t \mapsto f(x \wedge \bar{t}) \right) \cdot \left( t \mapsto \ad^{-m} g(y \wedge \theta \wedge \bar{t}) \right) \tag{$\dagger$}\label{eq:freeLoop:prime}\\
        =\ & \left( t \mapsto \ad f(x) (t) \right) \cdot \left( t \mapsto \ad \left( \ad^{-m} g(y \wedge \theta) \right)(t) \right) \\
        &\qquad \cdot \left( t \mapsto \ad f(x)(\bar{t}) \right) \cdot \left( t \mapsto \ad \left( \ad^{-m} g ( y \wedge \theta)\right) (\bar{t})\right) \\
        =\ &  t \mapsto C' \left( \ad f(x), \ad \left( \ad^{-m} g (y \wedge \theta) \right) \right)(t) \\
        \simeq \ & t \mapsto \left( \left( S' \comp \left( \ad f \wedge \ad\left( \ad^{-m} g \right) \right) \right)(x \wedge y \wedge \theta) \right) (t) \\
        =\ &   t \mapsto \ad^{-1}\left( S' \comp \left( \ad f \wedge \ad\left( \ad^{-m} g \right) \right) \right) (x \wedge y \wedge \theta \wedge t) \\
        =\ &   t \mapsto \ad^{-1}\left( S' \comp \left( \ad f \wedge \ad\left( \ad^{-m} g \right) \right) \right) (x \wedge y \wedge t \wedge \theta) \tag{$\dagger^\prime$}\label{eq:freeLoop:daggerPrime} \\
        =\ &   t \mapsto \ad^m \Big( \ad^{-1} \left( S' \comp \left( \ad f \wedge \ad \left( \ad^{-m} g \right) \right) (x \wedge y \wedge t) \right) \Big) (\theta).
    \end{align*}
    Note that we have implicitly used the homeomorphism $S^1 \wedge S^m \rightarrow S^m \wedge S^1$ and its inverse, respectively, in \eqref{eq:freeLoop:prime} and \eqref{eq:freeLoop:daggerPrime}, and thus, we do not need to pre-compose by this map in the homotopy group. Now, the last function is a loop in $\Omega^m Z$. Since $i_*$ is monic (see \eqref{i_*IsMonic}), we conclude that $\left\{ f, g \right\}_s$ is represented by the map 
    \[\ad^m \Big( \ad^{-1} \left( S \comp \left( \ad f \wedge \ad \left( \ad^{-m} g \right) \right) \right) \Big).\]
    But this map also represents $\ad^m \left( [f, \ad^{-m} g] \right)$, which completes the proof.
\end{proof}}

\section{Vanishing of the Brace Product}\label{sec:vanishingOfBrace}
In this section, we discuss the various consequences and equivalent criteria for the vanishing of both the James brace product and the generalized brace product.

\subsection{\texorpdfstring{$H$}{H}-splitting of the loop space of the total space}
\begin{thm}\label{generalized_brace-splitting}
    Suppose $F\overset{i}{\hookrightarrow} E\xrightarrow{p} B$ is a fibration with homotopy section $s:B\to E$. Then the generalized brace product $\{,\}_s$ vanishes if and only if the map $\varphi_s=(\Omega s\cdot\Omega i): \Omega B\times \Omega F\to \Omega E$ is an $H$-splitting.
\end{thm}
\begin{proof}
\sloppy
    We know from Proposition~\ref{map_inducing_brace} that $\{,\}_s$ vanishes if and only if $J_s:\Omega B\wedge \Omega F\to \Omega F$ is null-homotopic, since $J_s = \ad \left\{ \ad^{-1} \mathrm{Id}_{\Omega B}, \ad^{-1} \mathrm{Id}_{\Omega F} \right\}_s$. As $(\Omega i)_*$ is monic, $J_s$ is null-homotopic if and only if $(\Omega i)_*(J_s)=S\comp (\Omega s\wedge \Omega i)$ is null-homotopic. We have the following diagram
    \begin{equation*}
        \begin{tikzcd}
            \Omega B\times \Omega F\ar[rr, "q'"] \ar[d, "\Omega s\times \Omega i"] & &\Omega B\wedge \Omega F\ar[d, "\Omega s\wedge \Omega i"]\\
            \Omega E\times \Omega E \ar[urr, phantom, "\scalebox{1.7}{$\circlearrowright$}" description] \ar[rr, "q"] \ar[dr, "C", swap] & &\Omega E\wedge \Omega E\ar[dl, "S"]\\
            \ar[urr, phantom, "\cca{h}"{xshift=1em, yshift=6pt}] &\Omega E &
        \end{tikzcd}
    \end{equation*}
    In the above diagram, the square commutes (denoted by \scalebox{1.7}{$\circlearrowright$}), but the triangle commutes up to homotopy (denoted by $\cca{h}$). It then follows from the diagram and by an application of Lemma~\ref{cofib_seq}, that $S\comp (\Omega s\wedge \Omega i)$ is null homotopic if and only if $C\comp (\Omega s\times  \Omega i)$ is null-homotopic, or equivalently, $\Omega s\cdot\Omega i\simeq \Omega i\cdot\Omega s$. On the other hand, $\varphi_s$ is an $H$-map if and only if $\Omega s \cdot \Omega i\simeq \Omega i \cdot \Omega s$. Hence, $\varphi_s$ is an $H$-splitting if and only if $\{,\}_s$ vanishes.
\end{proof} 

Using Example 4.3 of Porter \cite{Porter}, we obtain a fibration, where the James brace product vanishes identically, but the generalized brace product does not.
\begin{eg}\label{eg:brace_vanish_but_not_genaralized}
Consider the fundamental class $\iota\in H^{6n}(K(\Z, 6n); \Z_3)\cong \text{Hom}(\Z, \Z_3)$ (i.e., $\iota$ corresponds to the quotient homomorphism from $\Z$ to $\Z_3$) and $\mathcal{P}^1_3(\iota)\in H^{6n+4}(K(\Z, 6n); \Z_3)$, where $\mathcal{P}^i_p:H^n(X;\Z_p)\to H^{n+2i(p-1)}(X;\Z_p)$ denotes the reduced $p^{\text{th}}$ power operation. Porter \cite[Def. 3.12, Lemma 3.13]{Porter} claims that $\iota, \mathcal{P}^1_3(\iota)$ are significant, i.e., there exists $\Tilde{\alpha}: \Sigma A\to K(\Z,6n)$ and $\Tilde{\beta}: \Sigma B\to K(\Z, 6n)$ for some $A, B$ with finitely generated homology such that $\Tilde{\alpha}^*(\iota)\neq 0$ and $\Tilde{\beta}^*(\mathcal{P}^1_3(\iota))\neq 0$. In particular, one can take $\Tilde{\alpha}:\Sigma S^{6n-1}\to K(\Z, 6n)$ as the generator of $\pi_{6n}(K(\Z,6n))$. Let $E$ be the total space of the pullback fibration of the path space fibration of $K(\Z_3, 12n+4)$ by 
$$\zeta:=\iota\smile \mathcal{P}^1_3(\iota)\in H^{12n+4}(K(\Z,6n);\Z_3)\cong \text{Hom}\left(K(\Z,6n), K(\Z_3,12n+4)\right).$$ 
Suppose $\alpha:\Sigma S^{6n-1}\to E$ and $\beta:\Sigma B\to E$ are lifts of $\Tilde{\alpha}$ and $\Tilde{\beta}$, respectively. Porter \cite[Eg. 4.3]{Porter} showed that the generalized Whitehead product, $[\beta,\alpha]\neq 0$; but as $E$ has only two non-zero homotopy groups in dimension $6n$ and $12n+3$ (follows from the long exact sequence), based on the dimensions of possible non-zero spherical Whitehead products it follows that all spherical products are zero in $E$. Now consider the free loop space fibration $\Omega E\hookrightarrow \mathcal{L}E \xrightarrow{ev} E$ with constant section $s$. Then Theorem~\ref{brace_whitehead_in_free_loop} for $m=1$ implies that the James brace product $\{,\}_s\equiv 0$, but the generalized brace product $\{\beta,\ad(\alpha)\}_s = \ad([\beta,\alpha])\neq 0$.
\end{eg}

\begin{remark}\label{remark:Kallel-Sjerve}
    Lemma 2.12 of Kallel-Sjerve \cite{Kallel-Sjerve} states that 
    \begin{quote}
        For a fibration $F\to E\to B$ with section $s$, if the James brace product vanishes identically, then $\varphi_s$ is an $H$-space splitting.
    \end{quote}
    In light of Theorem \ref{generalized_brace-splitting}, Example \ref{eg:brace_vanish_but_not_genaralized} provides an explicit counter-example to the above lemma. Although this lemma was included there as a general discussion on the properties of the brace product, it was not used anywhere in the paper.
    In the MathSciNet review of Liu's article \cite{Liu}, the first author, Sadok Kallel, remarked that this lemma is not true for arbitrary fibrations, but it is true when both the fibre and the base spaces are wedges of spheres. We observe the same in Corollary~\ref{base_fibre_product-wedge_sphere}.
\end{remark}

\subsection{The Samelson Product on the Loop Space of the Total Space}
We now derive an equivalent condition for the vanishing of the brace product in terms of the Samelson product.
\begin{thm}\label{phi_alg_map}
    Suppose $F\overset{i}{\hookrightarrow} E\overset{p}{\to} B$ is a fibration with a homotopy section $s:B\to E$. Then the James brace product $\{,\}_s$ vanishes for the given fibration if and only if the map $(\varphi_{s})_*:\pi_*(\Omega B\times\Omega F)\to \pi_*(\Omega E)$ is an algebra isomorphism with algebra structure given by the Samelson product.
\end{thm}
\begin{proof}
    For any space $X$, suppose $q_X:\Omega X\times \Omega X\to \Omega X\wedge \Omega X$ is the quotient map, $C_X:\Omega X\times \Omega X\to \Omega X$ is the commutator map, and $S_X:\Omega X\wedge \Omega X\to \Omega X$ is the Samelson map. We know that $S_X\comp q_X\simeq C_X$.\\
    \indent First note that the James brace product $\{,\}_s$ vanishes identically if and only if the top row $\Omega i\comp J_s\comp (\alpha\wedge \beta)$ in Diagram \eqref{diag:samelson1} is null-homotopic for any $\alpha\in \pi_k(\Omega B)\cong \pi_{k+1}(B),\beta\in \pi_l(\Omega F)\cong \pi_{l+1}(F)$. Since $q:S^k\times S^l\to S^k\wedge S^l$ is homotopy monic (i.e., $q^*$ is injective, by Lemma~\ref{cofib_seq}), $\Omega i\comp J_s\comp (\alpha\wedge \beta)$ is null homotopic if and only if $\Omega i\comp J_s\comp (\alpha\wedge \beta)\comp q$ is null-homotopic. But from the homotopy commutativity (denoted by $\cca{h}$) of Diagram \eqref{diag:samelson1}, this is equivalent to the null homotopy of the bottom row $C_E\comp (\Omega s\times \Omega i)\comp (\alpha \times \beta)=\Omega s(\alpha)\cdot\Omega i(\beta)\cdot\Omega s(\Bar{\alpha})\cdot\Omega i(\Bar{\beta})$. 
    \begin{equation}\label{diag:samelson1}
        \begin{tikzcd}
            S^k\wedge S^l \ar[rr, "{\alpha\wedge\beta}"]& & \Omega B\wedge \Omega F\ar[r, "J_s"] \ar[d,"\Omega s\wedge \Omega i"] &\Omega F\ar[r, "\Omega i"] & \Omega E\ar[d, equal]\\
            & &\Omega E\wedge \Omega E\ar[rr,"S_E"] \arrow[ur, phantom, "\cca{h}"{xshift=3em, yshift=3pt} description]& & \Omega E\\
            S^k\times S^l \arrow[uurr, phantom, "\scalebox{1.7}{$\circlearrowright$}" description]\ar[uu,"q"]\ar[r, "\alpha\times \beta"] &\Omega B\times \Omega F \ar[r, "\Omega s\times \Omega i"] & \Omega E \times \Omega E \arrow[urr, phantom, "\cca{h}"{xshift=0pt, yshift=3pt} description] \ar[u, "q_{E}"]\ar[rr, "C_E"] & &\Omega E \ar[u, equal]
        \end{tikzcd}
    \end{equation}
        This means that the James brace product $\{,\}_s$ vanishes identically if and only if
        \begin{equation}\label{brace_vanish_implies}
            \Omega s(\alpha)\cdot\Omega i(\beta)\simeq \Omega i(\beta)\cdot\Omega s(\alpha)\text{ for any }\alpha\in \pi_k(\Omega B),\beta\in \pi_l(\Omega F), k,l\geq 0.
        \end{equation}
        On the other hand, $(\varphi_s)_*$ is always a module isomorphism (see Proposition~\ref{homotopy_equivalence}). Thus, $\varphi_s$ is an algebra isomorphism if and only if $\varphi_s$ preserves the Samelson product. But, $\varphi_s$ preserves the Samelson product if and only if 
        \begin{equation}\label{eq:SamelsonMap}
            \varphi_s \comp S_{B\times F}\comp (f\wedge g)\simeq S_E\comp (\varphi_s\wedge \varphi_s)\comp (f \wedge g)\text{ for }f \in \pi_k(\Omega B\times \Omega F),\ g \in \pi_l(\Omega B\times \Omega F).
        \end{equation}
        Note that, in Diagram \eqref{diag:samelson2}, the marked squares and the triangle commute or commute up to homotopy. Since the quotient map $q: S^k\times S^l\to S^k\wedge S^l$ is homotopy monic, \eqref{eq:SamelsonMap} is equivalent to the homotopy commutativity of the dashed-dotted marked arrows in the Diagram \eqref{diag:samelson2}.
    \begin{equation}\label{diag:samelson2}
        \begin{tikzcd}
            S^k\wedge S^l\ar[r,"f\wedge g"]  &(\Omega B\times \Omega F)\wedge (\Omega B\times \Omega F) \ar[r, "S_{B\times F}"] &[1em] \Omega B\times \Omega F \ar[r,"\varphi_s", densely dashdotted]& \Omega E\ar[d, equal]\\
            S^k\times S^l \ar[u, "q"] \ar[d, "q"] \arrow[ur, phantom, "\scalebox{1.7}{$\circlearrowright$}" description] \ar[r, "f\times g", densely dashdotted]& (\Omega B\times \Omega F)\times (\Omega B\times \Omega F) \ar[u,"{q_{B\times F}}"] \ar[d,"{q_{B\times F}}"] \ar[r, "\varphi_s\times \varphi_s", densely dashdotted] \arrow[ur,phantom, "\cca{h}"{xshift=-3.5em, yshift=3pt}] \ar[ru, "{C_{B\times F}}"{xshift=8pt}, densely dashdotted] &\Omega E\times \Omega E \ar[d,"q_E"]\ar[r,"C_E", densely dashdotted] & \Omega E\\
            S^k\wedge S^l \ar[ur, phantom, "\scalebox{1.7}{$\circlearrowright$}" description] \ar[r,"f\wedge g"]& (\Omega B\times \Omega F)\wedge (\Omega B\times \Omega F) \arrow[ur, phantom, "\cca{h}"{xshift=3pt, yshift=3pt} description] \ar[r, "\varphi_s\wedge \varphi_s"] & \Omega E\wedge \Omega E \arrow[ur, phantom, "\cca{h}"{xshift=3pt, yshift=3pt} description] \ar[r,"S_E"] & \Omega E\ar[u, equal]
        \end{tikzcd}
    \end{equation}
    Thus, we only need to show that $\varphi_s\comp C_{B\times F}\comp (f\times g)\simeq C_E\comp (\varphi_s\times \varphi_s)\comp (f\times g)$. If $f=(\alpha_1,\beta_1)$ and $g=(\alpha_2,\beta_2)$ as components, then expanding the above two maps, we have 
    \begin{align*}
        &\varphi_s\comp C_{B\times F}\comp (f\times g)\\
        =\ & \varphi_s(\alpha_1\alpha_2\Bar{\alpha}_1\Bar{\alpha}_2, \beta_1\beta_2\Bar{\beta}_1\Bar{\beta}_2)\\
        =\ & \Omega s(\alpha_1\alpha_2\Bar{\alpha}_1\Bar{\alpha}_2)\cdot \Omega i(\beta_1\beta_2\Bar{\beta}_1\Bar{\beta}_2)\\
        =\ & \Omega s(\alpha_1)\cdot\Omega s(\alpha_2)\cdot\Omega s(\Bar{\alpha}_1)\cdot \Omega s(\Bar{\alpha}_2)\cdot \Omega i(\beta_1)\cdot \Omega i(\beta_2)\cdot \Omega i(\Bar{\beta}_1)\cdot \Omega i(\Bar{\beta}_2),\\
        & C_E\comp (\varphi_s\times \varphi_s)\comp (f\times g)\\
        =\ & C_E(\Omega s(\alpha_1)\cdot\Omega i(\beta_1), \Omega s(\alpha_2)\cdot\Omega i(\beta_2))\\
        =\ &  \Omega s(\alpha_1)\cdot\Omega i(\beta_1)\cdot \Omega s(\alpha_2)\cdot\Omega i(\beta_2)\cdot\overline{\Omega s(\alpha_1)\cdot\Omega i(\beta_1)}\cdot \overline{\Omega s(\alpha_2)\cdot\Omega i(\beta_2)}\\
        =\ & \Omega s(\alpha_1)\cdot\Omega i(\beta_1)\cdot\Omega s(\alpha_2)\cdot\Omega i(\beta_2)\cdot \Omega i(\Bar{\beta}_1)\cdot \Omega s(\Bar{\alpha}_1)\cdot\Omega i(\Bar{\beta}_2)\cdot \Omega s(\Bar{\alpha}_2)
    \end{align*}
    Now, if the James brace product $\{,\}_s$ vanishes identically, then using the homotopy \eqref{brace_vanish_implies} repeatedly, we see that the above two maps are homotopic, i.e., $\varphi_s$ is an algebra isomorphism.\\ 
    \indent Conversely, suppose $\varphi_s$ is an algebra isomorphism, or equivalently 
    $$\varphi_s\comp C_{B\times F}\comp (f\times g)\simeq C_E\comp (\varphi_s\times \varphi_s)\comp (f\times g)\text{ for all }f \in \pi_k(\Omega B\times \Omega F),\ g \in \pi_l(\Omega B\times \Omega F).$$
    For $\alpha\in \pi_k(\Omega B), \beta\in \pi_l(\Omega F)$, take $f:=(\alpha,c_F)\in \pi_k(\Omega B\times \Omega F),\ g:=(c_B,\beta)\in \pi_l(\Omega B\times \Omega F)$, where $c_B,c_F$ are constant maps in $\Omega B, \Omega F$, respectively. 
    Then, 
    $$\varphi_s\comp C_{B\times F}\comp (f\times g)\simeq C_E\comp (\varphi_s\times \varphi_s)\comp (f\times g)$$ 
    implies
    \begin{align*}
        &c_E\simeq \Omega s(\alpha)\cdot\Omega i(\beta)\cdot \Omega s(\Bar{\alpha})\cdot \Omega i(\Bar{\beta})\text{ ($c_E$ is the constant map in $\Omega E$)}\\
        \text{i.e., }& \Omega s(\alpha)\cdot\Omega i(\beta)\simeq \Omega i(\beta)\cdot \Omega s(\alpha).
    \end{align*}
    Therefore, \eqref{brace_vanish_implies} implies that the James brace product $\{,\}_s$ vanishes identically.
\end{proof}

\subsection{Decomposability of a fibration over a suspension}
The following is a direct generalization of \cite[Thm 2.1]{HiltonWhitehead} in the context of the generalized Whitehead product.
\begin{prop}\label{gen_whitehead_vanish_under_composition}
    Suppose $\alpha: \Sigma A\to X$ and $\beta: \Sigma B\to X$ are two maps such that $[\alpha,\beta]=0$. Then for any $f:\Sigma C\to \Sigma A$ and $g: \Sigma D\to \Sigma B$,
    $$[\alpha\comp f, \beta\comp g]=0.$$
\end{prop}
\begin{proof}
    Since $[\alpha,\beta] =0$, we know from \cite[Proposition 5.1]{Arkowitz_gen_Whitehead} that there exists a map $m:\Sigma A\times \Sigma B\to X$ such that $m\big|_{\Sigma A}\simeq \alpha$ and $m\big|_{\Sigma B}\simeq \beta$. Then, for any $f:\Sigma C\to \Sigma A$ and $g:\Sigma D\to \Sigma B$, the map $m\comp (f\times g):\Sigma A\times \Sigma B\to X$ satisfies 
    \begin{align*}
    \qquad m\comp (f\times g)\big|_{\Sigma A}\simeq\alpha\comp f\text{ and }m\comp (f\times g)\big|_{\Sigma B}\simeq \beta\comp g.
    \end{align*}
    Therefore, using the converse of \cite[Proposition 5.1]{Arkowitz_gen_Whitehead}, we conclude $[\alpha\comp f, \beta \comp g] =0$.
\end{proof}

The above proposition has the following useful corollary.
\begin{coro}\label{equi_Whitehead_brace}
If $\Sigma F\overset{i}{\hookrightarrow} E\xrightarrow{p} \Sigma B$ is a fibration with homotopy section $s$, then $[s,i]=i_*\{\mathrm{Id}_{\Sigma B}, \mathrm{Id}_{\Sigma F} \}_s=0$ if and only if the generalized brace product vanishes, i.e., $\{,\}_s\equiv 0$. 
\end{coro}
\begin{proof}
Let $[s,i] =0$. Then, for any $f:\Sigma X\to \Sigma B$ and $g:\Sigma Y\to \Sigma F$,  Proposition~\ref{gen_whitehead_vanish_under_composition} implies $[s\comp f,i\comp g] =0$, i.e., $i_*(\{f,g\}_s)=0$ (by definition of $\{,\}_s$). As $i_*$ is monic (see (\ref{i_*IsMonic})) and $f,g$ are arbitrary, we have $\{,\}_s\equiv 0$. Conversely, if $\{,\}_s\equiv 0$, then $\{\mathrm{Id}_{\Sigma B}, \mathrm{Id}_{\Sigma F} \}_s=0$, and hence, $[s,i] =i_* \{\mathrm{Id}_{\Sigma B}, \mathrm{Id}_{\Sigma F}\}_s =0$.
\end{proof}

Let us now establish a connection between the vanishing of the generalized brace product and the decomposability of the total space of a fibration. 
\begin{prop}\label{splitting_in_suspension_over_suspension}
Suppose $\Sigma F\overset{i}{\hookrightarrow} E\xrightarrow{p} \Sigma B$ is a fibration of CW-complexes with homotopy section $s$. If the generalized brace product $\{\mathrm{Id}_{\Sigma B}, \mathrm{Id}_{\Sigma F}\}_s =0$, then $E\simeq \Sigma B\times \Sigma F$, and moreover, $\varphi_s=(\Omega s\cdot\Omega i): \Omega \Sigma B\times \Omega \Sigma F\to \Omega E$ is an $H$-splitting.
\end{prop}
\begin{proof}
Note that, if $\{\mathrm{Id}_{\Sigma B},\mathrm{Id}_{\Sigma F}\}_s =0$, then $[s,i] =i_*(\{\mathrm{Id}_{\Sigma B},\mathrm{Id}_{\Sigma F}\}_s) = 0$. This implies (by \cite[Proposition 5.1]{Arkowitz_gen_Whitehead}) that there exists a map $m:\Sigma B\times \Sigma F\to E$ such that $m\big|_{\Sigma B}\simeq s$ and $m\big|_{\Sigma F}\simeq i$. Thus, if $j_1:\Sigma B\to \Sigma B\times \Sigma F$ and $j_2:\Sigma F\to \Sigma B\times \Sigma F$ denote the respective inclusions, then $m\comp j_1\simeq s$ and $m\comp j_2\simeq i$. 
We will show that $m$ is a homotopy equivalence. Since all our spaces are CW-complexes, by Whitehead's theorem, it is enough to prove that $m$ induces isomorphism on homotopy groups.
Now, for $\alpha\in \pi_*(\Sigma B\times \Sigma F)\cong \pi_*(\Sigma B)\oplus \pi_*(\Sigma F)$, we have $\alpha_1\in \pi_*(\Sigma B)$ and $\alpha_2\in \pi_*(\Sigma F)$ such that $\alpha=(j_1)_*(\alpha_1)+(j_2)_*(\alpha_2)$. Then 
\begin{align*}
    (p\comp m)_*(\alpha)&= (p\comp m)_*((j_1)_*(\alpha_1)+(j_2)_*(\alpha_2))\\
    &= (p\comp m\comp j_1)_*(\alpha_1)+(p\comp m\comp j_2)_*(\alpha_2))\\
    &=(p\comp s)_*(\alpha_1)+(p\comp i)_*(\alpha_2)\\
    &=\alpha_1+0 \text{ (as $(p\comp s)_*=1$ and $(p\comp i)_*=0$)}\\
    &=(p_1)_*(\alpha)
\end{align*}
Here, $p_1:\Sigma B\times\Sigma F\to \Sigma B$ is the projection onto the first coordinate. Thus, we have the following commutative diagram of two short exact sequences.
\[
\begin{tikzcd}
    0\ar[r]& \pi_*(\Sigma F) \ar[d,equal]\ar[r,"(j_2)_*"] &\pi_*(\Sigma B\times \Sigma F) \ar[d,"m_*"] \ar[r,"(p_1)_*"] &\pi_*(\Sigma B) \ar[r]\ar[d,equal]&0\\
    0\ar[r]& \pi_*(\Sigma F) \ar[r,"i_*"] & \pi_*(E)\ar[r, "p_*"] &\pi_*(\Sigma B)\ar[r]&0
\end{tikzcd}
\]
Hence, by the five lemma, $m_*$ is an isomorphism, and therefore $E \simeq \Sigma B \times \Sigma F$. Moreover, since $\{\mathrm{Id}_{\Sigma B}, \mathrm{Id}_{\Sigma F}\}_s = 0$, Corollary~\ref{equi_Whitehead_brace} implies that $\{,\}_s \equiv 0$. Consequently, by Theorem~\ref{generalized_brace-splitting}, $\varphi_s$ is an $H$-splitting.
\end{proof}

As an immediate corollary, we have the following interesting result.

\begin{coro}\label{brace_in_sphere_over_sphere}
Suppose $S^m\overset{i}{\hookrightarrow} E\xrightarrow{p}S^n$ is a fibration. Then $E\simeq S^n\times S^m$ if and only if there exists a homotopy section $s$ such that the James brace product $\{\mathrm{Id}_{S^n},\mathrm{Id}_{S^m}\}_s=0$.
\end{coro}
\begin{proof}
If there exists a homotopy section $s$ of the given fibration such that the James brace product $\{\mathrm{Id}_{S^n},\mathrm{Id}_{S^m}\}_s=0$, then from above Proposition~\ref{splitting_in_suspension_over_suspension}, we have $E\simeq S^n\times S^m$.

Conversely, if $E \simeq S^n \times S^m$, it is established in \cite[Corollary 1.1]{Sasao-HomotopyTypeOfSphericalFibreSpacesOverSpheres} that $E$ is fibre homotopically trivial; that is, there exists a homotopy equivalence $h: S^n \times S^m \to E$ such that $p \comp h \simeq p_1$ and $m \comp j_2 \simeq i$, where $p_1 : S^n \times S^m \to S^n$ is the projection onto the first coordinate and $j_2 : S^m \to S^n \times S^m$ denotes the inclusion into the second coordinate. Define $s := h \comp j_1$, where $j_1 : S^n \to S^n \times S^m$ is the inclusion into the first coordinate. Observe that $p \comp s = p \comp h \comp j_1 \simeq p_1 \comp j_1 = \text{Id}_{S^n}$, so $s$ serves as a homotopy section for the fibration. Furthermore,
\begin{align*}
    i_*\{\text{Id}_{S^n}, \text{Id}_{S^m}\}_s&= [s,i] =[m\comp j_1, m\comp j_2]\\
    &=m_*([j_1,j_2])=m_*([(1,0), (0,1)]=0
\end{align*}
Therefore, $\{\text{Id}_{S^n}, \text{Id}_{S^m}\}_s=0$, and hence $s$ is a section for the fibration for which the brace $\{\text{Id}_{S^n}, \text{Id}_{S^m}\}_s=0$.
\end{proof}

\begin{eg}
In the context of oriented $S^2$ bundles over $S^2$, there are only two bundles, one is the trivial $S^2\times S^2$ and the other is the connected sum $\mathbb{P}^2\#\Bar{\mathbb{P}}^2$. Here, $\mathbb{P}^2$ denotes the complex projective plane, and $\Bar{\mathbb{P}}^2$ denotes $\mathbb{P}^2$ with opposite orientation. As the intersection forms of the above two $4$-manifolds are $\left(\begin{array}{cc}
    0 & 1 \\
    1 & 0
\end{array}\right)$ and $\left(\begin{array}{cc}
    1 & 0 \\
    0 & -1
\end{array}\right)$, respectively (see \cite[pp. 21-22]{Kirby}), and they are non-isomorphic with integer coefficients, we get $S^2\times S^2\not\simeq \mathbb{P}^2\#\Bar{\mathbb{P}}^2$. Hence, we have the following corollary.
\begin{coro}
There does not exist any homotopy section $s$ of the non-trivial $S^2$ bundle over $S^2$ with respect to which the brace product vanishes.
\end{coro}
\end{eg}
Let us now deduce some equivalent criteria for the vanishing of the brace product in some interesting classes of fibrations.

\begin{prop}\label{BraceInWedgeOverWedgeSpaces}
Suppose $ F\overset{i}{\hookrightarrow} E\xrightarrow{p} B$ is a fibration with homotopy section $s$, where $B=\vee_{r\in\mathcal{I}} \Sigma^2 B_r$ and $F=\vee_{j\in\mathcal{J}} \Sigma^2 F_j$ are finite wedges of double suspensions, with inclusions $\alpha_r :\Sigma^2 B_r \hookrightarrow B$ and $\beta_j :\Sigma^2 F_j \hookrightarrow F$. Then the following are equivalent.
\begin{enumerate}[\quad(a)]
    \item $[s,i] =0$. 
    \item $\{,\}_s\equiv 0$.
    \item $\{\alpha_r, \beta_j\}_s=0$ for all $(r,j) \in \mathcal{I} \times \mathcal{J}$.
\end{enumerate}
\end{prop}
\begin{proof}
    Since $B$ and $F$ are suspensions themselves, the equivalence of $(a)$ and $(b)$ follows from Corollary~\ref{equi_Whitehead_brace}. Therefore, it is enough to prove the equivalence of $(b)$ and $(c)$. 
    If we assume $\{,\}_s\equiv 0$, then $\{\alpha_r, \beta_j\}_s=0$ for all $(r,j)\in \mathcal{I}\times \mathcal{J}$ is immediate.
    Conversely, suppose $\{\alpha_r, \beta_j\}_s=0$ for all $(r,j)\in \mathcal{I}\times \mathcal{J}$. Now, to prove $\{,\}\equiv 0$, it is enough to prove $[s,i]=0$ by Corollary~\ref{equi_Whitehead_brace}. Note that, $s, i$ can be written as $s=\sum_{r\in \mathcal{I}} s_r$ and $i=\sum_{j\in \mathcal{J}}i_j$, where $s_r, i_j$ are the following compositions:
    \begin{align*}
        s_r:B=\vee_{r\in \mathcal{I}}\Sigma^2 B_r\xrightarrow{p_r^B}\Sigma^2 B_r\overset{\alpha_r}{\hookrightarrow} B\xrightarrow{s} E\\
        i_j:F=\vee_{j\in \mathcal{J}}\Sigma^2 F_j\xrightarrow{p_j^F}\Sigma^2 F_j\overset{\beta_j}{\hookrightarrow} F\xrightarrow{i} E
    \end{align*}
    Here, $p_r^B, p_j^F$ are the projection maps to $r$-th and $j$-th component of $B$ and $F$, respectively.
    Since $s_r$ and the $i_j$'s (in particular, $s$ and $i$) are finitely many maps from double suspension spaces, by the bi-linearity of the generalized Whitehead product \cite[Proposition 3.4]{Arkowitz_gen_Whitehead}, we have
    \begin{align*}
        [s,i]&=\left[\sum_{r\in \mathcal{I}}s_r, \sum_{j\in \mathcal{J}}i_j\right]=\sum_{(r,j)\in \mathcal{I}\times \mathcal{J}}[s_r, i_j]\\
        &=\sum_{(r,j)\in \mathcal{I}\times \mathcal{J}}[s\comp \alpha_r \comp p_r^B, i\comp \beta_j\comp p_j^F]
    \end{align*}
    Now, since $[s\comp \alpha_r, i\comp \beta_j]=i_*\{\alpha_r, \beta_j\}_s=0$ for all $(r,j)\in \mathcal{I}\times \mathcal{J}$, then by Proposition~\ref{gen_whitehead_vanish_under_composition}, we have $[s\comp \alpha_r \comp p_r^B, i\comp \beta_j\comp p_j^F]=0$ for all $(r,j)\in \mathcal{I}\times \mathcal{J}$. Hence 
    \[[s,i]=\sum_{(r,j)\in \mathcal{I}\times \mathcal{J}}[s\comp \alpha_r \comp p_r^B, i\comp \beta_j\comp p_j^F]=0.\qedhere\]
\end{proof}
\begin{remark}
    In general, the bi-linearity of the Whitehead product does not behave well under infinite sums, as noted in \cite{Brazas-IdentitiesForWhitehead}. For this reason, we have restricted our discussion to finite wedges. If the bi-linearity issue could be resolved, the proposition could potentially be extended to infinite wedges as well.
\end{remark}
If the suspension spaces in the above proposition are spheres of suitable dimensions, the following corollary follows from Theorem \ref{generalized_brace-splitting}.
\begin{coro}\label{base_fibre_product-wedge_sphere}
    Suppose $ F \overset{i}{\hookrightarrow} E\xrightarrow{p} B$ is a fibration with a homotopy section $s$, where $F=\vee_{r=1}^m S^{k_r}$ and $B=\vee_{j=1}^n S^{l_j}$ for $k_r, l_j\geq 2$. Then the map $\varphi_s=(\Omega s\cdot\Omega i):\Omega B\times \Omega F\to \Omega E$ is an $H$-splitting if and only if the James brace product $\{\alpha_j,\beta_r\}_s=0$ for $1\leq r\leq m$, $1\leq j\leq n$, and $\alpha_j:S^{l_j}\hookrightarrow B$, $\beta_r:S^{k_r}\hookrightarrow F$ are the inclusion maps.
\end{coro}

\subsection{Fibrations with vanishing brace product}
The following proposition allows us to construct a new fibration with vanishing (generalized) brace product from any given fibration that admits a section.
\begin{prop}\label{brace_vanish_bundle}
Suppose $F\overset{i}{\hookrightarrow} E\xrightarrow{p} B$ is a fibration with a homotopy section $s:B\to E$. Also, assume there exists a product space $\Tilde{B}=\prod_k B_k$ and a map $f: \Tilde{B}\to B$ such that $f\comp \iota_k:B_k\to B$ is null-homotopic for each $k$, where $\iota_k: B_k\hookrightarrow \Tilde{B}$ is the inclusion map. Then the pullback fibration $F\overset{j}{\hookrightarrow} f^*E\xrightarrow{\tilde{p}}\Tilde{B}$ has the homotopy section $\Tilde{s}=f^*s$ such that
$\{,\}_{\tilde{s}}\equiv 0.$
\end{prop}
\begin{proof}
Let $\alpha \in \left[\Sigma X, \Tilde{B}\right]$ and $\beta\in \left[\Sigma Y, F\right]$. As $\Tilde{B}=\prod_k B_k$, $\left[\Sigma X, \Tilde{B}\right]=\prod_k \left[\Sigma X, B_k\right]$ and thus $\alpha=\prod_k \alpha_k$, where $\alpha_k= \iota_k\comp p_k\comp \alpha$ for $p_k:\Tilde{B}\to B_k$ is the projection onto the $k^{\text{th}}$ coordinate. Then as $f\comp \iota_k\simeq 0$ for all $k$,
\begin{align*}
    f_*(\alpha)&= \prod_k f_*(\alpha_k)
    = \prod_k f_*(\iota_k\comp p_k\comp \alpha)
    = \prod_k f\comp \iota_k\comp p_k\comp \alpha
    = 0
\end{align*}
Hence, from Lemma~\ref{pullback_brace}, we have 
$$\{\alpha,\beta\}_{\Tilde{s}}=\{\alpha,\beta\}_{f^*s}=\{f_*(\alpha,\beta)\}_s=\{0,\beta\}_s=0.$$
Since the above equality holds for any $\alpha\in [\Sigma X, \Tilde{B}]$ and $\beta\in [\Sigma Y, F]$, we have $\{,\}_{\Tilde{s}}\equiv 0$.
\end{proof}
\begin{coro}\label{product_sphere_brace}
Suppose $E$ is an $F$-bundle over a sphere $S^{l_1+\cdots+l_k}$ with homotopy section $s$. Also, let $f: S^{l_1}\times \cdots \times S^{l_k}\to S^{l_1+\cdots+l_k}$ be the pinching map (or any continuous map). Then $\{,\}_{f^*s}=0$ in the pullback bundle $f^*E$ over $S^{l_1}\times \cdots \times S^{l_k}$.
\end{coro}
\begin{proof}
Note that, if $\iota_j:S^{l_j}\hookrightarrow S^{l_1}\times \cdots \times S^{l_k}$ denotes the inclusion for $j=1,2,\ldots, k$, then $f\comp \iota_j: S^{l_j}\to S^{l_1+l_2+\cdots+ l_k}$ is null homotopic because $\pi_{l_j}(S^{l_1 + l_2 +\cdots+ l_k}) = 0$ for all $j=1,2,\ldots, k$. Hence, the corollary follows from previous Proposition~\ref{brace_vanish_bundle}.
\end{proof}

Here is an example of a fibration for which the brace product vanishes, but the fibration is not trivial. This non-triviality is detected by the second Stiefel-Whitney class. 

\begin{eg}\label{eg:S^nBundleOverGenusSurface}
Consider any connected, oriented $S^n$ bundle $E$ over the genus $g$ surface $\Sigma_g$ with section $s$, where $g\geq 1$ and $n\geq 2$. Let $j:\Sigma_g^{(1)}\hookrightarrow \Sigma_g$ be the inclusion of the 1-skeleton $\Sigma_g^{(1)}=\vee_{n=1}^{2g}S^1$ of $\Sigma_g$. Then the pullback bundle $j^*(E)$ is an oriented bundle over the wedge of circles with section $j^*(s)$. Hence, it is trivial, as the fibre of $j^*(E)$ is $S^n$. Let $j^*(s)=(s_1,s_2): \Sigma_g^{(1)}\to \Sigma_g^{(1)}\times S^n$. But being $n\geq 2$, $\pi_1(S^n)=0$ and consequently, $s_2\simeq 0$. Thus, by Proposition~\ref{BraceInWedgeOverWedgeSpaces}, $\{,\}_{j^*(s)}\equiv 0$. Now, for $\alpha\in \pi_1(\Sigma_g)$, we may assume $\alpha=j_*(\alpha')$ for $\alpha'\in \pi_1(\Sigma_g^{(1)})$ (by cellular approximation). Then, for $\beta\in \pi_k(S^n)$, by Lemma~\ref{pullback_brace}, $\{\alpha,\beta\}_s=\{j_*(\alpha'), \beta\}_s=\{\alpha',\beta\}_{j^*(s)}=0$. But as $\Sigma_g$ does not have any higher homotopy groups, the James brace product $\{,\}_s\equiv 0$.

Suppose the oriented $S^n$ bundle $E$ over $\Sigma_g$ is the unit sphere bundle $S(\zeta)\xrightarrow{\pi} \Sigma_g$ of some real vector bundle $\zeta$ of rank $n+1$ over $\Sigma_g$ with structure group $SO(n+1)$. The equivalence class of such $\zeta$ corresponds bijectively to $\pi_1(SO(n+1))\cong \Z_2$, meaning that, up to equivalence, there are only two such bundles. One is the trivial bundle $\R^{n+1}\times \Sigma_g$, and the other is $f^*(\gamma^1_1\oplus \epsilon^{n-1})$, where $f:\Sigma_g\to \C\mathbb{P}^1(\cong S^2)$ is a degree $1$ map. Here, $\gamma_1^1$ and $\epsilon^{n-1}$ denote the tautological line bundle and the rank $n-1$ trivial bundle over $\C\mathbb{P}^1$, respectively. The non-equivalency of the above two bundles is captured by the second Stiefel-Whitney class $w_2$. Indeed, $w_2(\R^{n+1}\times \Sigma_g)=0$ but 
\begin{align*}
    w_2\left(f^*(\gamma^1_1\oplus \epsilon^{n-1})\right)&=f^*w_2\left(\gamma^1_1\oplus \epsilon^{n-1}\right)\\
    &=f^*w_2(\gamma^1_1)\neq 0,
\end{align*} 
because by construction, $w_2(\gamma^1_1)\neq 0$ and $f$ is a degree $1$ map, inducing an isomorphism on $H^2$.
Thus, for non-triviality, let us assume $\zeta$ is the equivalence class of $f^*(\gamma^1_1\oplus \epsilon^{n-1})$. In particular, $w_2(\zeta)\neq 0$.\\
\indent Let us calculate the total Stiefel-Whitney class $w(S(\zeta)):=w(T(S(\zeta)))$ of the tangent space of $S(\zeta)$, which is a homotopy invariant for $S(\zeta)$. 
Furthermore, the normal bundle of the inclusion $S(\zeta)\to \zeta$ is a trivial rank $1$ bundle. 
By horizontal and vertical decomposition, we have $T\zeta\cong T\Sigma_g\oplus \zeta$. Thus, 
$$T(S(\zeta))=T\zeta\big|_{S(\zeta)}\cong \pi^*(T\Sigma_g) \oplus \pi^*(\zeta).$$
Note that the Euler class of $T\Sigma_g$ is a multiple of the Euler characteristic $\chi(\Sigma_g)=2-2g$ and hence $w_2(T\Sigma_g)=0$. 
Also, the orientability of $\Sigma_g$ implies $w_1(T\Sigma_g)=0$. 
Consequently, $w(T\Sigma_g)=1$.
Therefore,
\begin{align*}
   w(T(S(\zeta)))&= w(\pi^*(T\Sigma_g) \oplus \pi^*(\zeta))\\
    &= \pi^*w(T\Sigma_g) \smile \pi^*w(\zeta)\\ 
    &= 1 \smile \pi^*w(\zeta)\\
    &= \pi^*w(\zeta).
\end{align*}
 As $\zeta$ is orientable, $w_1(\zeta)=0$. 
 But as $\pi^*$ is injective in $H^2$ and $w_2(\zeta)(\neq 0)\in H^2(\Sigma_g)$, $\pi^*w_2(\zeta)\neq 0$. Also, see that
 $$w(T(S^n\times \Sigma_g))=w(TS^n\oplus T\Sigma_g)=w(TS^n)\smile w(T\Sigma_g)=1\smile 1=1.$$
 Hence,
 $w(T(S(\zeta)))=1+\pi^*w_2(\zeta)\neq 1=w(T(S^n\times \Sigma_g))$ i.e.,
 $E=S(\zeta)\not \simeq S^n\times \Sigma_g$. 
\end{eg}

In the following example, we consider the free loop space fibrations over $S^2$. This provides an example of a fibration where the generalized brace product vanishes, and thus by Theorem~\ref{generalized_brace-splitting}, we obtain the $H$-splitting of the corresponding based loop space fibration.
\begin{eg}\label{Eg:FreeLoopSpaceFibrationOfS2} 
    Let us consider the fibration $\Omega^m S^2 \hookrightarrow \mathcal{L}_0^m S^2 \rightarrow S^2$ with the canonical section $s$, which sends each point of $S^2$ to the constant map that assigns this point on $S^m$. For $m = 1$, consider the maps $f = \mathrm{Id}_{S^2}$ and $g = \ad \mathrm{Id}_{S^2}$. Then, by  Theorem~\ref{brace_whitehead_in_free_loop}
    \[\left\{ f, g \right\}_s = \ad[f, \ad^{-1} g] = \ad[\mathrm{Id}_{S^2}, \mathrm{Id}_{S^2}] = 2 \ad \gamma \ne 0,\]
    where $\gamma: S^3 \rightarrow S^2$ is the Hopf map. On the other hand, for $m \ge 2$, the generalized brace product in this fibration vanishes. In other words, as a consequence of Theorem~\ref{brace_whitehead_in_free_loop}, it is enough to show that all the generalized Whitehead products 
    \[[\Sigma A, S^2] \times [\Sigma^{m+1} B, S^2] \rightarrow [\Sigma^{m+1} A\wedge B, S^2]\]
    vanish. Now, from the Hopf fibration $S^1 \hookrightarrow S^3 \xrightarrow{\gamma} S^2$, we have the long exact sequence
    \[\dots \rightarrow \left[ \Sigma B, \Omega^m S^1 \right] \rightarrow \left[ \Sigma B, \Omega^m S^3 \right] \xrightarrow{\gamma_*} \left[ \Sigma B, \Omega^m S^2 \right] \rightarrow \left[ \Sigma B, \Omega^{m-1} S^1 \right] \rightarrow \dots \]
    Since $\Omega^i S^1$ is contractible for $i \ge 2$ and is homotopy equivalent to $\mathbb{Z}$ for $i = 1$, it follows that $\left[ \Sigma B, \Omega^i S^1 \right] = 0$ for $i \ge 1$. In particular, for $m \ge 2$, we have 
    \[\left[ \Sigma^{m+1} B, S^2 \right] = \left[ \Sigma B, \Omega^m S^2 \right] \overset{\gamma_*}{\cong} \left[ \Sigma B, \Omega^m S^3 \right] = \left[ \Sigma^{m+1} B, S^3 \right].\]
    Say, $f : \Sigma A \rightarrow S^2 , g : \Sigma^{m+1} B \rightarrow S^2$ are given. Using the above identification, we then have $\tilde{g} : \Sigma^{m+1} B \rightarrow S^3$ such that $g \simeq \gamma \comp \tilde{g}$. Now, it follows from \cite{Whitehead-ProductsInHomotopyGroups} that $[\mathrm{Id}_{S^2}, \gamma] = 0$. But then by Proposition~\ref{gen_whitehead_vanish_under_composition}, we have $[f, g] = [\mathrm{Id}_{S^2} \comp f, \gamma \comp \tilde{g}] = 0$. Consequently, the generalized brace product vanishes identically.
\end{eg}

\begin{remark}
    Hansen asked in \cite{Hansen} whether $\mathcal{L}_0^2 S^2 \simeq S^2 \times \Omega^2 S^2$. As a consequence of the above example and Theorem~\ref{generalized_brace-splitting}, we have $\Omega \mathcal{L}_0^m S^2 \simeq \Omega S^2 \times \Omega^{m+1} S^2$ as $H$-spaces for $m \ge 2$, and consequently the (quasi) Lie algebras $\left( \pi_*\left( \mathcal{L}_0^2 S^2 \right), [~,~] \right)$ and $\left( \pi_*\left( S^2 \times \Omega^2 S^2 \right), [~,~] \right)$ are abstractly isomorphic. Thus, to detect whether the spaces are homotopy equivalent or not, we possibly need stronger invariants.
\end{remark}

\section{Localization of Spaces and Brace Product} \label{sec:genBraceLocalization}
Let us recall some relevant definitions for localization of spaces. Suppose $X$ is a simply connected CW-type space and $\Pa$ is a set of some primes. Consider 
$$K = K_{\mathcal{P}}:=\{m/k:m,k\in \mathbb{Z},\ \text{gcd}(k,p)=1\ \forall p\in \Pa \}\overset{\text{sub-ring}}{\subseteq}\Q.$$

\begin{defi}\label{defn:PLocalSpace}
    A simply connected space $X$ is called \emph{$\mathcal{P}$-local}, if $\pi_i(X)$ is a $K$-module for all $i$.
\end{defi}

\begin{defi}\label{defn:PLocalEquivalence}
    Given two simply connected spaces $X$ and $Y$, a map $f: X \to Y$ is called a \emph{$\mathcal{P}$-local equivalence} if $\pi_*(f) \otimes 1_K : \pi_\ast(X)\otimes K \rightarrow  \pi_\ast(Y) \otimes K$ is an isomorphism.
\end{defi}

Equivalent definitions of $\mathcal{P}$-local spaces and $\mathcal{P}$-local equivalences in terms of reduced homology follow from \cite[Thm 9.3]{Felix_rational} and \cite[Thm 9.6]{Felix_rational} respectively.

\begin{defi}\label{defn:PLocalization}
    Given a simply connected space $X$, a \emph{$\mathcal{P}$-localization} of $X$ is given by a map $\varphi_X : X \rightarrow X_{\mathcal{P}}$, where $X_{\mathcal{P}}$ is a $\mathcal{P}$-local space, and $\varphi_X$ is a $\mathcal{P}$-local equivalence.
\end{defi}

In case $\Pa=\phi$, we have $K=\Q$ and a $\Pa$-localization is called a \textit{rational space of $X$}. It is denoted by $X_\Q$ instead of $X_\phi$. For any $\Pa$, it follows from \cite[Thm 9.7]{Felix_rational} that given any simply connected CW complex, there is a relative CW complex $(X_{\mathcal{P}}, X)$, such that $X_{\mathcal{P}}$ is $\mathcal{P}$-local and the inclusion $X \hookrightarrow X_{\mathcal{P}}$ is a $\mathcal{P}$-localization. Moreover, if $f: X\to Z$ is a continuous map, where $Z$ is any $\Pa$-local space, then $f$ can be extended to a unique (up to homotopy) map $\Tilde{f}: X_\Pa\to Z$ (see \cite[Thm 9.7(ii)]{Felix_rational}). As a consequence, we have a canonical isomorphism $[S^n_{\mathcal{P}}, X_{\mathcal{P}}] \cong \pi_n(X_{\mathcal{P}})$. Furthermore, $[S^n_\mathcal{P}, X_\mathcal{P}]$ can be given a $K$-module structure which makes the bijection a $K$-module isomorphism. A $\Pa$-localization of $X$ is unique up to homotopy equivalence relative to $X$. Subsequently, the following results will be useful in our case.
\begin{lemma}{\cite[Cor 3.2, Proposition 3.3]{MimuraNishidaToda-Localization}}\label{prop_local_spaces}
Suppose $X$ and $Y$ are two simply connected spaces, with $\mathcal{P}$-localizations $\varphi_X : X \hookrightarrow X_{\mathcal{P}}$ and $\varphi_Y : Y \hookrightarrow Y_{\mathcal{P}}$. Then, the following maps are $\mathcal{P}$-localizations.
\begin{enumerate}[\quad(a)]
    \item $\varphi_X \vee \varphi_Y : X \vee Y \hookrightarrow X_\Pa\vee Y_\Pa$
    \item $\Sigma \varphi_X : \Sigma X \hookrightarrow \Sigma (X_{\mathcal{P}})$
    \item $\Omega \varphi_X : \Omega X \hookrightarrow \Omega \left( X_{\mathcal{P}} \right)$, provided $X$ is $2$-connected.
    \item $\varphi_X \times \varphi_Y : X \times Y \hookrightarrow X_{\mathcal{P}} \times Y_{\mathcal{P}}$
    \item $\varphi_X \wedge \varphi_Y : X \wedge Y \hookrightarrow X_\Pa \wedge Y_\Pa$.
\end{enumerate}
\end{lemma}

Mimura-Nishida-Toda \cite[Thm 3.1]{MimuraNishidaToda-Localization} also showed that localization preserves fibrations and cofibrations up to homotopy. In our case, we will be particularly interested in rationalization and use the following convention throughout this section: \textit{when discussing rationalization (or localization), we will assume that all spaces are simply connected, even if this is not stated explicitly}. 

For a given fibration $F\overset{i}{\hookrightarrow} E\overset{p}{\to} B$ with a section $s$, consider rational spaces $F_\Q,E_\Q, B_\Q$ with rationalized maps $i_\Q,p_\Q, s_\Q$ so that 
$$F_\Q\overset{i_\Q}{\rightarrow} E_\Q\xrightarrow{p_\Q}B_\Q$$ 
is a fibration with homotopy section $s_\Q$ and has the following homotopy commutative diagram
\begin{equation}\label{Diagram:rationalization}
    \begin{tikzcd}
    F \arrow[r, "i", hook] \arrow[d, "\varphi_F"] & E \arrow[r, "p"'] \arrow[d, "\varphi_E"] & B \arrow[d, "\varphi_B"] \arrow[l, bend right=30,"s"']\\
    F_\Q \arrow[r,"i_\Q"]  & E_\mathbb{Q} \arrow[r, "p_\mathbb{Q}"]   & B_\mathbb{Q} \arrow[l, bend left=30, "s_\Q"] 
\end{tikzcd}
\end{equation}
We will refer to $F_\Q\overset{i_\Q}{\rightarrow} E_\Q\xrightarrow{p_\Q}B_\Q$ as \emph{the rationalized fibration} of the given fibration $F\overset{i}{\hookrightarrow} E\overset{p}{\to} B$.

The following is an example of a fibration with vanishing James brace product in the rationalized fibration setting.
\begin{eg}\label{eg:oddSphereBundleJamesVanish}
Consider the rationalized fibration of any $S^n$ bundle over a simply connected space $X$ with section $s$, where $n>1$ is odd. As the homotopy groups of $S^n_\Q$ are concentrated in degree $n$, the only relevant brace product is $\{\alpha,\beta\}_{s_\Q}$ for $\alpha\in \pi_l(X_\Q), \beta\in \pi_n(S^n_\Q)$. Note that $\{\alpha,\beta\}_{s_\Q}\in \pi_{n+l-1}(S^n_\Q)=0$ unless $l=1$. But as $\pi_1(X_\Q)=0$, we have $\{,\}_{s_\Q}\equiv 0$.
\end{eg}

More generally, one could take any fibration over a simply connected space $X$ with fibre $K(G,n)$, an Eilenberg-Maclane space, and a section $s$. Then a similar argument implies that $\{,\}_{s_{\mathbb{Q}}}\equiv 0$. Such fibrations always appear in Postnikov towers of spaces.

In general, if the base and fiber of a fibration are rationally equivalent to a product of simply connected rational suspension spaces, then the vanishing of the generalized brace can be determined using the following proposition.
\begin{prop}\label{prop:RationalBraceInProductSpaces}
    Suppose $ F\overset{i}{\hookrightarrow} E\xrightarrow{p} B$ is a fibration of simply connected spaces and it admits a homotopy section $s$. Also, suppose 
    $$B_\Q\underset{\simeq}{\xrightarrow{f_B}} \prod_{r\in\mathcal{I}} \Sigma (B_r)_\Q\text{ and }F_\Q\underset{\simeq}{\xrightarrow{f_F}} \prod_{j\in\mathcal{J}} \Sigma (F_j)_\Q$$ 
    with homotopy inverses $g_B, g_F$ respectively and $\mathcal{I}, \mathcal{J}$ are some finite indexing sets. If
    \begin{align*}
        \alpha_r&:\Sigma (B_r)_\Q \hookrightarrow \prod_{r\in\mathcal{I}} \Sigma (B_r)_\Q \xrightarrow{g_B} B_\Q \text{ and}\\
        \beta_j&:\Sigma (F_j)_\Q \hookrightarrow \prod_{j\in\mathcal{J}} \Sigma (F_j)_\Q \xrightarrow{g_F} F_\Q
    \end{align*}
    are the composite maps, then the following are equivalent.
    \begin{enumerate}[\quad(a)]
        \item $\{,\}_{s_\Q}\equiv 0$.
        \item $\{\alpha_r, \beta_j\}_{s_\Q}=0$ for all $(r,j) \in \mathcal{I} \times \mathcal{J}$.
    \end{enumerate}
\end{prop}
\begin{proof}
    If we assume $\{,\}_{s_\Q}\equiv 0$, then it follows immediately that $\{\alpha_r, \beta_j\}_{s_\Q}=0$ for all $(r,j)\in \mathcal{I}\times \mathcal{J}$.
    Conversely, suppose $\{\alpha_r, \beta_j\}_{s_\Q}=0$ for all $(r,j)\in \mathcal{I}\times \mathcal{J}$. 
    To show that the brace product $\{,\}_{s_\Q}$ vanishes, it suffices - by Remark~\ref{remark:UniversalBrace} - to verify that $\{f_0, g_0\}_{s_\Q}=0$, or equivalently, that $[s_\Q\comp f_0, i_\Q\comp g_0]=0$, where $f_0:\Sigma\Omega B_\Q\to B_\Q$ and $g_0:\Sigma\Omega F_\Q\to F_\Q$ are the adjoints of the identity maps on $\Omega B_\Q$ and $\Omega F_\Q$, respectively. Since $B$ and $F$ are simply connected, both $\Sigma\Omega B_\Q$ and $\Sigma\Omega F_\Q$ are double suspension spaces (as rational suspension spaces are homotopy equivalent to wedges of spheres \cite[Thm 24.5]{Felix_rational}). 
    Since $g_B\comp f_B\simeq \mathrm{Id}_{B_\Q}$, we obtain
    \begin{align*}
        s_\Q\comp f_0&\simeq s_\Q\comp g_B\comp f_B\comp f_0\\
        &\simeq \sum_{r\in \mathcal{I}} s_\Q \comp \alpha_r\comp f_0^r, 
    \end{align*}
    where $f_0^r$ is the $r$-th component of the composite $f_B\comp f_0:\Sigma \Omega B_\Q\to \prod_{r\in\mathcal{I}} \Sigma (B_r)_\Q$. Note that the above summation is taken in the generalized homotopy group, which is well-defined because the domain is a double suspension, making the group commutative. Similarly, 
    $$i_\Q\comp g_0\simeq \sum_{j\in \mathcal{J}} i_\Q\comp \beta_j\comp g_0^j,$$
    where $g_0^j$ is the $j$-th component of the composite $f_F\comp g_0:\Sigma\Omega F_\Q\to \prod_{j\in\mathcal{J}} \Sigma (F_j)_\Q$.
    Therefore, by the bi-linearity of the generalized Whitehead product, 
    $$[s_\Q\comp f_0, i_\Q\comp g_0]=\sum_{(r,j)\in \mathcal{I}\times \mathcal{J}}[s_\Q\comp \alpha_r\comp f_0^r, i_\Q\comp \beta_j\comp g_0^j].$$
    Since $[s_\Q\comp \alpha_r, i_\Q\comp \beta_j]=(i_\Q)_*\{\alpha_r, \beta_j\}_{s_\Q}=0$, it follows from Proposition~\ref{gen_whitehead_vanish_under_composition} that $[s_\Q\comp \alpha_r\comp f_0^r, i_\Q\comp \beta_j\comp g_0^j]=0$ for each $(r,j)\in \mathcal{I}\times\mathcal{J}$. Therefore, we conclude that  $[s_\Q\comp f_0, i_\Q\comp g_0]=0$, and consequently, the brace product $\{,\}_{s_\Q}$ vanishes identically for the rationalized fibration.
\end{proof}

Using the above proposition, we obtain the following example of a rationalized fibration with a vanishing generalized brace product. 
\begin{eg}\label{eg:LieGroup}
Consider the rationalized fibration of any $S^n$-bundle over a simply connected compact Lie group $G$ with a section $s$, where $n > 1$. For example, one may take the unit sphere bundle of the tangent bundle of $G$, suitably embedded in some Euclidean space, ensuring the existence of a section. It is known \cite[Cor 1, \S 3, Chap V]{Serre-Groupes} that $G$ is rationally homotopy equivalent to a product of odd-dimensional spheres. Specifically, there exists a homotopy equivalence $f:G_\Q\to S^{k_1}_\Q \times S^{k_2}_\Q\times \cdots \times S^{k_r}_\Q$ with homotopy inverse $g:S^{k_1}_\Q \times S^{k_2}_\Q\times \cdots \times S^{k_r}_\Q \to G_\Q$,  
where each $k_i$ is an odd number greater than $1$. Let 
$$\alpha_l:\Sigma S^{k_l-1}_\Q \simeq S^{k_l}_\Q \hookrightarrow S^{k_1}_\Q \times S^{k_2}_\Q\times \cdots \times S^{k_r}_\Q \xrightarrow{g} G_\Q$$ 
be the composite map for $l=1,2,...,r$. Now, observe that $\{\alpha_l, \mathrm{Id}_{S^n_\Q}\}_{s_\Q}=0$ for each $l$. Indeed,
$$\{\alpha_l, \mathrm{Id}_{S^n_\Q}\}_{s_\Q}\in [\Sigma (S^{n-1}_\Q\wedge S^{k_l-1}_\Q), S^n_\Q]=[S^{n+k_l-1}_\Q, S^n_\Q]=\pi_{n+k_l-1}(S^n_\Q)=0,$$  
since each $k_l$ is an odd number greater than $1$ and $\pi_{n+k_l-1}(S^n_\Q)=0$ unless $k_l=1$ with $n$ odd, or $k_l=1, n$ with $n$ even. Hence, by the above proposition, the generalized brace $\{,\}_{s_\Q}$ vanishes. Therefore, by Theorem~\ref{generalized_brace-splitting}, the based loop space fibration of this rationalized fibration splits as $H$-spaces.
\end{eg}

In the above example, it is not clear whether the James brace product vanishes at the level of the original fibration (before rationalization). We now present an example where the brace product is non-vanishing for the original fibration, but it vanishes in the rationalized fibration.
\begin{eg}
    Consider the trivial fibration $S^{n} \rightarrow S^{n} \times S^{n} \rightarrow S^{n}$ for odd $n$, with section the diagonal map $\Delta: S^n\to S^n\times S^n$. Then, $\left\{ \mathrm{Id}_{S^{n}}, \mathrm{Id}_{S^{n}}\right\}_\Delta =\left[\mathrm{Id}_{S^{n}}, \mathrm{Id}_{S^{n}}\right] \in \pi_{2n-1}(S^{n})$. By graded commutativity of the Whitehead product, we know that $2\left[\mathrm{Id}_{S^{n}}, \mathrm{Id}_{S^{n}}\right]=0$. Moreover, Gilmore \cite{Gilmore-SomeWhiteheadOnOddSpheres} proved that if $n+1\neq 2^r$ for any $r$, then $\left[\mathrm{Id}_{S^{n}}, \mathrm{Id}_{S^{n}}\right]$ is a non-zero $2$-torsion element, i.e., in that case the James brace product is non-vanishing. However, the James brace product in the rationalized fibration is zero, as $\pi_{2n-1}^\mathbb{Q}(S^{n}) = 0$. Note that the same argument works (using Corollary~\ref{brace_in_sphere_over_sphere}) for any fibration $S^{2k+1} \to E \to S^n$ with a section, for which the total space is \emph{not} split and $n>1$.
\end{eg}

At the level of localized spaces, we obtain the following equivalence relating the vanishing of the generalized Whitehead product. This will be useful to prove the next theorem.
\begin{lemma}\label{RationalGenWhiteAndWhitehead}
    Suppose $A$ is a simply connected space with $r: A\to A_\Pa$ as its $\Pa$-localization. Also, suppose $\alpha\in [\Sigma X, A_\Pa]$ and $\beta\in [\Sigma Y, A_\Pa]$. Then for the unique elements $\alpha_\Pa\in [\Sigma X_\Pa, A_\Pa]$ and $\beta_\Pa \in [\Sigma Y_\Pa, A_\Pa]$, the generalized Whitehead product $[\alpha_\Pa, \beta_\Pa]=0$ whenever $[\alpha,\beta]=0$. 
\end{lemma}
\begin{proof}
    Since $[\alpha,\beta]=0$, by \cite[Thm 5.1]{Arkowitz_gen_Whitehead}, there exists a map $m:\Sigma X\times \Sigma Y\to A_\Pa$ such that
    $m|_{\Sigma X\vee \Sigma Y}\simeq \alpha\vee \beta.$
    Now employing Lemma~\ref{prop_local_spaces}(a), and (d), we can choose localization maps $r': \Sigma X\vee \Sigma Y\to \Sigma X_{\mathcal{P}}\vee\Sigma Y_\mathcal{P}$, and $r'':\Sigma X\times \Sigma Y\to \Sigma X_{\mathcal{P}}\times\Sigma Y_\mathcal{P}$ together with $m_\Pa: \Sigma X_{\mathcal{P}}\times \Sigma Y_\mathcal{P}\to A_\Pa$ so that the following diagram commutes up to homotopy.
    \begin{equation*}
        \begin{tikzcd}[column sep=large]
            \Sigma X\vee \Sigma Y \arrow[r, hook] \arrow[d, "r'"] \arrow[rr, "\alpha\vee\beta", bend left]   & \Sigma X\times \Sigma Y \arrow[d, "r''"] \arrow[r, "m"]   & A_\mathcal{P} \\
            \Sigma X_{\mathcal{P}}\vee\Sigma Y_\mathcal{P} \arrow[r, hook] \arrow[rru, "\alpha_\mathcal{P}\vee\beta_\mathcal{P}", bend right=49]   & \Sigma X_{\mathcal{P}}\times \Sigma Y_\mathcal{P} \arrow[ru, "m_\Pa"']   & 
        \end{tikzcd}
    \end{equation*}
    Note that the homotopy commutativity of the top-most triangle is obtained from the given condition, and that of the bottom triangle is obtained via the uniqueness of $\alpha_\Pa\vee \beta_\Pa$ (see \cite[Thm 9.7(ii)]{Felix_rational}). Since the map $m_\Pa: \Sigma X_\Pa\times \Sigma Y_\Pa\to A_\Pa$ satisfies $m_\Pa|_{\Sigma X_\Pa\vee \Sigma Y_\Pa}\simeq \alpha_\Pa\vee \beta_\Pa$, by the converse of \cite[Thm 5.1]{Arkowitz_gen_Whitehead}, we conclude $[\alpha_\Pa, \beta_\Pa]=0$.
\end{proof}

Let us recall that for a well-pointed, path connected space $Y$, there is a rational homotopy equivalence between $\Sigma Y$ and  $\vee_{\alpha\in I}S^{n_\alpha}$ for $n_\alpha\geq 2$ and some indexing set $I$ (see \cite[Thm 24.5]{Felix_rational}).
Using this fact, the hypothesis of Proposition~\ref{splitting_in_suspension_over_suspension} can be relaxed to require only the vanishing of the James brace product in a rationalized fibration whose base and fibre are suspension spaces, as follows.

\begin{thm}\label{cor:JamesBraceVanishImpliesRationalEquiv}
    Suppose $\Sigma F_\Q \xrightarrow{i_\Q} E_\Q \xrightarrow{p_\Q} \Sigma B_\Q$ is the rationalized fibration with section $s_\Q$ of a fibration $\Sigma F \xrightarrow{i} E \xrightarrow{p} \Sigma B$ with section $s$, where $\Sigma F, E,$ and $\Sigma B$ are simply connected finite CW complexes. If the James brace product $\{,\}_{s_\Q}$ vanishes, then $ \Sigma B_\Q\times \Sigma F_\Q\simeq E_\Q$.
\end{thm}
\begin{proof}
    Using Proposition~\ref{splitting_in_suspension_over_suspension}, it is enough to prove that $\{\mathrm{Id}_{\Sigma B_\Q}, \mathrm{Id}_{\Sigma F_\Q}\}_{s_\Q}=0$ or, equivalently, $[s_\Q, i_\Q]=(i_\Q)_*\{\mathrm{Id}_{\Sigma B_\Q}, \mathrm{Id}_{\Sigma F_\Q}\}_{s_\Q}=0$. Since $F$ and $B$ are finite CW complexes, their suspensions $\Sigma F$ and $\Sigma B$ are also finite CW complexes. Then, by \cite[Thm 24.5]{Felix_rational} and Lemma~\ref{prop_local_spaces}, we have homotopy equivalences 
    \begin{align*}
        &f:\vee_{r=1}^n\Sigma^2 S^{n_r}_\Q \xrightarrow{\simeq} \Sigma B_\Q, ~n_r\geq 0, \\
        &g:\vee_{j=1}^m\Sigma^2 S^{m_j}_\Q \xrightarrow{\simeq} \Sigma F_\Q, ~m_j\geq 0.
    \end{align*}
    Therefore, using Proposition~\ref{gen_whitehead_vanish_under_composition}, it reduces to prove $[s_\Q\comp f, i_\Q\comp g]=0$. 
    Suppose 
    \begin{align*}
        &\alpha_r:\Sigma^2 S^{n_r}_\Q\hookrightarrow \vee_{r=1}^n\Sigma^2 S^{n_r}_\Q,\\
        &\beta_j: \Sigma^2 S^{m_j}_\Q \hookrightarrow \vee_{j=1}^m\Sigma S^{m_j}_\Q
    \end{align*}
    are inclusions, and 
    \begin{align*}
        &\Tilde{\alpha}_r:\Sigma^2 S^{n_r}\to \vee_{r=1}^n\Sigma^2 S^{n_r}_\Q,\\
        &\Tilde{\beta}_j:\Sigma^2 S^{m_j}\to \vee_{j=1}^m\Sigma^2 S^{m_j}_\Q
    \end{align*}
    are restriction of $\alpha_r$ and $\beta_j$ respectively so that $(\Tilde{\alpha}_r)_\Q\simeq \alpha_r$ and $(\Tilde{\beta}_j)_\Q\simeq \beta_j$. 
    Since the James brace product vanishes identically, we have 
    \begin{align*}
        \{f \comp \Tilde{\alpha}_r, g \comp \Tilde{\beta}_j\}_{s_\Q}&=0\\
        \text{i.e., } [s_\Q \comp f \comp \Tilde{\alpha}_r, i_\Q \comp g \comp \Tilde{\beta}_j]&=(i_\Q)_*\{f \comp \Tilde{\alpha}_r, g \comp \Tilde{\beta}_j\}_{s_\Q}=0.
    \end{align*}
    Then Lemma~\ref{RationalGenWhiteAndWhitehead} implies $[s_\Q \comp f \comp \alpha_r, i_\Q \comp g \comp \beta_j]=0$. Again by Proposition~\ref{gen_whitehead_vanish_under_composition}, we have $[s_\Q \comp f \comp \alpha_r\comp \pi_r^B, i_\Q \comp g \comp \beta_j\comp \pi_j^F]=0$, where
    \begin{align*}
        &\pi_r^B:\vee_{r=1}^n\Sigma^2 S^{n_r}_\Q \rightarrow \Sigma^2 S^{n_r}_\Q,\\
        &\pi_j^F:\vee_{j=1}^m\Sigma S^{m_j}_\Q \rightarrow \Sigma^2 S^{m_j}_\Q
    \end{align*}
    are the projections. Then note that $s_\Q\comp f=\sum_{r=1}^n s_\Q\comp f\comp \alpha_r\comp \pi_r^B$ and $i_\Q\comp g=\sum_{j=1}^m i_\Q\comp g\comp \beta_j\comp \pi_j^F$. Consequently, using bi-linearity of generalized Whitehead product, we have
    \begin{align*}
        [s_\Q\comp f, i_\Q\comp g]&= \left[ \sum_{r=1}^n s_\Q\comp f\comp \alpha_r\comp \pi_r^B, \sum_{j=1}^m i_\Q \comp g\comp \beta_j\comp \pi_j^F \right]\\
        &=\sum_{r,j=1}^{n,m}\left[s_\Q \comp f \comp \alpha_r\comp \pi_r^B, i_\Q \comp g \comp \beta_j\comp \pi_j^F\right]=0
    \end{align*}
    Hence, the proof follows.
\end{proof}

\begin{remark}\label{remark:ConverseOfRationalDecomposition}
    \begin{enumerate}[(i)]
        \item The above theorem can be extended to the case where the base and fibre spaces have the homotopy type of simply connected finite CW complexes that are co-$H$-spaces, rather than merely suspension spaces. To do so, we use the result that simply connected finite co-$H$-spaces are rationally homotopy equivalent to a wedge of spheres \cite{Henn-OnAlmostRationalCoHSpaces}. The rest of the proof proceeds in essentially the same way.
        \item The converse of the above statement does not hold in its current form. That is, there exist sphere bundles over spheres whose total space is rationally equivalent to the product of the base and fiber, even though the original fibration does not admit a section and, consequently, has no brace product.

        \sloppy For example, consider the principal $S^3$ (or $SU(2)$) bundle $SU(3)$ over $S^5 \cong SU(3)/SU(2)$. Since $SU(3)$ is not homeomorphic to the product $S^3 \times S^5$ (as $\pi_4(SU(3)) = 0 \neq \mathbb{Z}_2 = \pi_4(S^3 \times S^5)$), this principal bundle does not admit any section. However, by inspecting the Serre spectral sequence for this fibration, one finds that the cohomology of $SU(3)$ is isomorphic to that of $S^3 \times S^5$, showing that $SU(3)$ is rationally equivalent to $S^3 \times S^5$.
    \end{enumerate} 
\end{remark}

\section{Brace Product and the \texorpdfstring{$J$}{J}-homomorphism}\label{Sec:J-HomoAndBrace}

Whitehead \cite{Whitehead-OnTheHomotopy} defined a group homomorphism $J: \pi_{n-1}(SO(q))\to \pi_{n+q-1}(S^q)$ for $q,n-1\geq 2$. We generalize the notion of $J$-homomorphism and prove similar results as in \cite{James-Whitehead, Stocker-NoteOnQuasi} for the generalized brace product.

\subsection{Generalized \texorpdfstring{$J$}{J}-Homomorphism}\label{GenJ-HomoDef}
Suppose $Map_\bullet(\Sigma Y, \Sigma Y)$ consists of base point preserving continuous maps from $\Sigma Y$ to itself. We define a map
\[J : [X, Map_\bullet(\Sigma Y, \Sigma Y)] \rightarrow [\Sigma X \wedge Y, \Sigma Y]\]
as follows. Given a map $\rho : X \rightarrow Map_\bullet(\Sigma Y, \Sigma Y)$, consider the map
\begin{align*}
    \tilde{\rho} : X \star Y = C X \times Y \cup X \times CY &\rightarrow \Sigma Y \\
    \left( [x,t],y \right) & \mapsto \bullet \\
    \left( x, [y, t] \right) & \mapsto \rho({x}) \left( [y, t] \right)
\end{align*}
Here, $X\star Y$ denotes the reduced topological join of $X$ and $Y$, and it is identified with $C X \times Y \cup X \times CY$. We then define $J[\rho] = [\tilde{\rho} \comp \mu]$, where $\mu: \Sigma X \wedge Y \rightarrow X \star Y$ is a choice of homotopy inverse for the quotient map $X \star Y \rightarrow \Sigma X \wedge Y$, which is known to be a homotopy equivalence.

\begin{remark}\label{Remark:GenJ-Homo}
    Let $Y = S^{q-1}$ and $X = S^{n-1}$. Viewing $\Sigma Y = S^q$ as $\mathbb{R}^q \cup {\infty}$, we obtain an inclusion
    $SO(q)\hookrightarrow Map_\bullet(\Sigma Y, \Sigma Y)=Map_\bullet(S^q, S^q)$ via the action that fixes $\infty$. Since $\Sigma X\wedge Y=S^{n+q-1}$, the restriction of the generalized $J$-homomorphism to $[S^{n-1}, SO(q)]$ gives a map from $\pi_{n-1}(SO(q))\to [S^{n+q-1}, S^q]=\pi_{n+q-1}(S^q)$. Furthermore, this map agrees with the classical $J$-homomorphism of Whitehead (see \cite[§2]{James-Whitehead}). 
\end{remark}

Next, observe that $Map_\bullet(\Sigma Y, \Sigma Y)$ is a homotopy associative $H$-space, with homotopy identity given by the constant map, and admitting homotopy inverses. Thus, the domain of $J$ is a group. Indeed,
\[\left[ X, Map_\bullet(\Sigma Y, \Sigma Y) \right] = \left[ X_{+}, Map_\bullet(\Sigma Y, \Sigma Y) \right] = \left[ X_{+} \wedge \Sigma Y, \Sigma Y \right] = \left[ \Sigma X_{+} \wedge Y, \Sigma Y \right]\]
is a group, where $X_+$ denotes $X$ with a disjoint base point adjoined. An Eckmann-Hilton type argument then leads to the following lemma.

\begin{lemma}\sloppy
    If $X = \Sigma X'$ is a suspension, then the two group structures on $[\Sigma X', Map_\bullet(\Sigma Y, \Sigma Y)]$, one induced by the co-$H$-space structure of $\Sigma X'$ and the other by the $H$-space structure of $Map_\bullet(\Sigma Y, \Sigma Y)$, coincide. Moreover, this group is abelian.
\end{lemma}

We also observe the following properties of the generalized $J$-homomorphism.
\begin{lemma}\label{J(epsilon)=0}
    \begin{enumerate}
        \item $J$ is a group homomorphism.
        \item Given any $\varphi \in Map_\bullet(\Sigma Y, \Sigma Y)$ and $\rho : X \rightarrow Map_\bullet(\Sigma Y, \Sigma Y)$, we have $J[\varphi \comp \rho] = \varphi_* J[\rho]$.
        \item If $X$ is a suspension, then $J[\varepsilon] = 0$, where $\varepsilon$ is the map $\varepsilon : x \mapsto  \mathrm{Id}_{\Sigma Y}$.
        \item If $X$ is a suspension, then given any $\phi \in Map_\bullet(\Sigma Y, \Sigma Y)$, $J[\varphi] = 0$, where $\varphi$ is the map $\varphi : x \mapsto \phi$.
    \end{enumerate}
\end{lemma}
\begin{proof}
    We use additive symbols for multiplication in both groups. Consider two maps $h_i : X \rightarrow Map_\bullet (\Sigma Y, \Sigma Y)$ representing some $\alpha_i$, for $i = 1, 2$. Then, $\alpha_1 + \alpha_2$ is represented by a map
    \begin{align*}
        h : X &\rightarrow Map_\bullet(\Sigma Y, \Sigma Y) \\
        x & \mapsto 
        \left( [y, t] \mapsto 
        \begin{cases}
            h_1(x)([y,2t]), \quad 0 \le t \le \frac{1}{2}\\
            h_2(x)([y,2t - 1]), \quad \frac{1}{2} \le t \le 1
        \end{cases} \right)
    \end{align*}
    We then have the map 
    \begin{align*}
        \tilde{h} : C X \times Y \cup X \times CY &\rightarrow \Sigma Y \\
        \left( [x,t], y \right) &\mapsto \bullet \\
        \left( x, [y,t] \right) &\mapsto 
        \begin{cases}
            h_1 \left( x \right) \left( [y,2t] \right), \quad 0 \le t \le \frac{1}{2} \\
            h_2 \left( x \right) \left( [y,2t - 1] \right), \quad \frac{1}{2} \le t \le 1
        \end{cases}
    \end{align*}
    Clearly, $J\left( \alpha_1 + \alpha_2 \right)$ is represented by $\tilde{h} \comp \mu$. On the other hand, $\tilde{h} \comp \mu$ represents $J[h_1] + J[h_2]$ as well. Hence, $J\left( \alpha_1 + \alpha_2 \right) = J\alpha_1 + J\alpha_2$, i.e., $J$ is a group homomorphism.

    Proof of (2) follows immediately from the definition of $J$.

    Next, suppose $X = \Sigma X^\prime$ is a suspension. Then, $J[\varepsilon] + J[\varepsilon] = [\tilde{\varepsilon} \comp \mu] + [\tilde{\varepsilon} \comp \mu]$ is represented by $\phi \comp \mu$, where
    \begin{align*}
        \phi : C \Sigma X^\prime \times Y \cup \Sigma X^\prime \times Y &\rightarrow \Sigma Y \\
        \left( [[x^\prime, s], t], y \right) & \mapsto \bullet \\
        \left( [x^\prime, s], [y, t] \right) & \mapsto 
        \begin{cases}
            [y, t], \quad 0 \le s \le \frac{1}{2} \\
            [y, t], \quad \frac{1}{2} \le s \le 1
        \end{cases}
    \end{align*}
    Clearly, $\phi = \tilde{\varepsilon}$, and thus 
    $J[\varepsilon] + J[\varepsilon] = J[\varepsilon] \Rightarrow J[\varepsilon] = 0$, which proves (3).

    Proof of (4) follows from (2) and (3).
\end{proof}

\subsection{Generalized Brace Product for a Fibration over a Suspension}\label{GenBraceOfaFibOverSuspension}
Let us consider a fibration $\Sigma F \hookrightarrow E \rightarrow \Sigma B$, where the characteristic map  $\rho: B \rightarrow Map_\bullet(\Sigma F, \Sigma F)$  satisfies $\rho(\bullet) = \mathrm{Id}_{\Sigma F}$. In particular, we have 
\[E = CB \times \Sigma F \cup_\rho CB \times \Sigma F,\]
given by the clutching construction, i.e.,
\begin{align*}
    E=((CB \times \Sigma F) \sqcup (CB \times \Sigma F))/\sim\\
    ([b,1], [f,t])\sim ([b,1], \rho(b)[f,t])
\end{align*} 
Observe that this fibration possesses a canonical section $s: \Sigma B \rightarrow E$, which will sometimes be referred to as $s_\infty$, and is defined by 
\begin{equation}\label{eq:section_s_infty}
    s([b,t]) = \left[ [b, t], \bullet \right]
\end{equation}
This is well-defined because $\rho$ fixes the base point of $\Sigma F$. We now relate the generalized brace product with the generalized $J$-homomorphism.

\begin{thm}\label{BraceAndJ-homoInSuspensionOverSuspension}
Suppose $\Sigma F \overset{i}{\hookrightarrow} E \rightarrow \Sigma B$ is a fibration as above with the canonical section $s$.
    \begin{enumerate}
        \item For $\varepsilon : B \rightarrow Map_\bullet(\Sigma F, \Sigma F)$ defined by $\varepsilon(b) = \mathrm{Id}_{\Sigma F}$, we have 
        $$\left\{ \mathrm{Id}_{\Sigma B}, \mathrm{Id}_{\Sigma F} \right\}_s = J[\varepsilon] - J[\rho].$$ 
        In particular, if $B$ is a suspension, then 
        $$\left\{ \mathrm{Id}_{\Sigma B}, \mathrm{Id}_{\Sigma F} \right\}_s = -J[\rho].$$

        \item For any base point preserving self homotopy equivalence $\varphi: \Sigma F \rightarrow \Sigma F$, we have 
        $$\left\{ \mathrm{Id}_{\Sigma B}, \varphi \right\}_s = J[\varphi \comp \varepsilon] -J[\rho].$$ 
        In particular, if $B$ is a suspension, then 
        $$\left\{ \mathrm{Id}_{\Sigma B}, \varphi \right\}_s = -J[\rho].$$
    \end{enumerate}
\end{thm}
\begin{proof}
    (1) Firstly, observe that the total space $E$ is given as the attaching space 
    \begin{equation}\label{eq:totalSpaceAttachingMap}
        E\simeq \left( CB \times CF \right) \cup_\phi \left( \Sigma B \vee \Sigma F \right),
    \end{equation}
    where $\phi$ is given as
    \begin{equation} \label{eq:totalSpaceAttachinMapFormula}
        \begin{aligned}
            \phi : CB \times F \cup B \times CF &\rightarrow \Sigma B \vee \Sigma F \\
            \left( [b,t], f \right) &\mapsto \left( [b,t], \bullet \right) \\
            \left( b, [f,t] \right) &\mapsto \left( \bullet, \rho(b)[f,t] \right)
        \end{aligned}
    \end{equation}
    Indeed, we have the diagram 
    \begin{equation} \label{eq:attachingMapEquivalenceDiagram}
        \begin{tikzcd}
            && \Sigma F \arrow{dd}{} \\
            B \times \Sigma F \arrow{r}{\hat{\rho}} \arrow[hook]{dd}[swap]{} \arrow[bend left=30]{rru}{\rho} & CB \times \Sigma F  \arrow{ru}{\simeq} \\
            && \tilde{E}\\
            CB \times \Sigma F \arrow{r}[swap]{} \arrow[bend left=30]{rru}{} & E \arrow[leftarrow,crossing over]{uu}{} \arrow[dashed]{ru}{\Phi}
        \end{tikzcd}
    \end{equation}
    Here, $\hat{\rho}$ is given by $(b,[f,t]) \mapsto ([b,1],\rho(b)[f,t])$. Both $E$ and $\tilde{E}$ are homotopy pushouts since the inclusion $B \times \Sigma F \hookrightarrow CB \times \Sigma F$ is a cofibration \cite[Prop. 6.2.6]{Arkowitz-IntroToHomotopy}. Hence, the deformation retract $CB \times \Sigma F \rightarrow \Sigma F$ collapsing $CB$ induces a homotopy equivalence $\Phi: E \rightarrow \tilde{E}$, making the above diagram (homotopy) commutative. But clearly, $\tilde{E} = \left( CB \times CF \right) \cup_\phi (\Sigma B \vee \Sigma F)$, which combined with $\Phi$ demonstrates the homotopy equivalence of \eqref{eq:totalSpaceAttachingMap}.
    
    Let us now consider the following maps.
    \begin{align*}
        \tilde{\rho} : CB \times F \cup B \times CF &\rightarrow \Sigma F \\
        \left( [b,t], f \right) &\mapsto \bullet  \\
        \left( b, [f,t] \right) &\mapsto \rho(b)[f,t] \\
        \epsilon : CB \times F \cup B \times CF &\rightarrow \Sigma F \\
        \left( [b,t], f \right) &\mapsto \bullet \\
        \left( b, [f,t] \right) &\mapsto [f,1-t] 
    \end{align*}
    We then have, $J[\rho] = [\tilde{\rho} \comp \mu], J[\varepsilon] = -[\epsilon \comp \mu]$. We also have the Whitehead map
    \begin{align*}
        \omega : CB \times F \cup B \times CF &\rightarrow \Sigma B \vee \Sigma F \\
        \left( [b,t], f \right) &\mapsto \left( [b,t], \bullet \right) \\
        \left( b, [f,t] \right) &\mapsto \left( \bullet, [f,t] \right)
    \end{align*}
    It follows from Remark~\ref{Gen_White_prop} that 
    \[i_*\left\{ \mathrm{Id}_{\Sigma B}, \mathrm{Id}_{\Sigma F} \right\}_s = [s, i] = [\left( s \vee i \right) \comp \omega] = (s \vee i)_*[\omega].\]
    We claim that 
    \begin{equation} \label{eq:attachinMapDecomposition}
        [\phi \comp \mu] = [\iota \comp \tilde{\rho} \comp \mu] + [\iota \comp \epsilon \comp \mu] + [\omega \comp \mu],
    \end{equation}
    where $\iota : \Sigma F \hookrightarrow \Sigma B \vee \Sigma F$ is the inclusion. Indeed, the right-hand side is represented by $\psi \comp \mu$, where $\psi$ is the following map.
    \begin{align*}
        \psi : CB \times F \cup B \times CF &\rightarrow \Sigma B \vee \Sigma F \\
        \left( [b,t],f\right) &\mapsto 
        \begin{cases}
            \left( \bullet, \bullet \right), \quad 0 \le t \le \frac{2}{3} \\
            \left( [b,3t - 2], \bullet \right), \quad \frac{2}{3} \le t \le 1
        \end{cases} \\
        \left( b, [f,t] \right) &\mapsto 
        \begin{cases}
            \left( \bullet, \rho(b) [f, 3t] \right), \quad 0 \le t \le \frac{1}{3} \\
            \left( \bullet, [f, 2 - 3t] \right), \quad \frac{1}{3} \le t \le \frac{2}{3} \\
            \left( \bullet, [f, 3t - 2] \right), \quad \frac{2}{3} \le t \le 1
        \end{cases}
    \end{align*}
    Consider the homotopy
    \begin{align*}
        \mathsf{h}_s : CB \times F \cup B \times CF &\rightarrow \Sigma B \vee \Sigma F \\
        \left( [b,t], f \right) & \mapsto 
        \begin{cases}
            (\bullet, \bullet) , \quad 0 \le t \le \frac{2(1 - s)}{3} \\
            \left( \left[ b, \frac{3t + 2s - 2}{1 + 2s} \right], \bullet \right), \quad \frac{2(1 - s)}{3} \le t \le 1
        \end{cases} \\
        \left( b, [f,t] \right) & \mapsto 
        \begin{cases}
            \left( \bullet, \rho(b) \left[ f,  \frac{3t}{1 + 2s}\right] \right), \quad 0 \le t \le \frac{1 + 2s}{3} \\
            \left( \bullet, \left[ f, 2  - 3t + 2s\right] \right), \quad \frac{1 + 2s}{3} \le t \le \frac{2 + s}{3} \\
            \left( \bullet, \left[ f, 3t - 2 \right] \right), \quad \frac{2 + s}{3} \le t \le 1
        \end{cases}
    \end{align*}
    We have $\mathsf{h}_0 = \psi, \mathsf{h}_1 = \phi$, and hence, $\psi \simeq \phi$, establishing (\ref{eq:attachinMapDecomposition}).
    
    Let us denote the induced map as
    \[j : \Sigma B \vee \Sigma F \rightarrow \left( CB \times CF \right) \cup_\phi \left( \Sigma B \vee \Sigma F \right).\]
    Then, we deduce that $j \comp \phi$ is null-homotopic. Also, it follows from the diagram (\ref{eq:attachingMapEquivalenceDiagram}) that 
    \[j \simeq \Phi \comp \left( s \vee i \right): \Sigma B \vee \Sigma F \hookrightarrow \tilde{E}, \quad  \text{and} \quad j \comp \iota \simeq \Phi \comp i: \Sigma F \hookrightarrow \tilde{E}.\] Thus, applying $j_*$ on both sides of \eqref{eq:attachinMapDecomposition}, we have
    \begin{align*}
        0 &= j_*[\iota \comp \tilde{\rho} \comp \mu] + j_*[\iota \comp \epsilon \comp \mu] + j_*[\omega \comp \mu] \\
        &= \Phi_* \left( i_*[\tilde{\rho} \comp \mu] + i_*[\epsilon \comp \mu] + [(s \vee i) \comp \omega \comp \mu] \right) \\
        &= \Phi_* \left( i_*\left( J[\rho] - J[\varepsilon] \right) + [s, i] \right) \; (\text{as $\Phi$ is a homotopy equivalence}) \\
        \Rightarrow \Phi_* [s, i] &= \Phi_* i_* \left( J[\varepsilon] - J[\rho] \right)
    \end{align*}
    Since $\Phi$ is a homotopy equivalence, and $i_*$ is monic, we have 
    $$\left\{ \mathrm{Id}_{\Sigma B}, \mathrm{Id}_{\Sigma F} \right\}_s = J[\varepsilon] - J[\rho]$$
    Furthermore, when $B$ is a suspension, Lemma~\ref{J(epsilon)=0} implies that 
    $$\left\{ \mathrm{Id}_{\Sigma B}, \mathrm{Id}_{\Sigma F} \right\}_s = J[\varepsilon] - J[\rho] = -J[\rho]$$

    (2) Let $\psi : \Sigma F \rightarrow \Sigma F$ be such that $\psi \comp \varphi \simeq \mathrm{Id}_{\Sigma F} \simeq \varphi \comp \psi$. Consider the attaching map
    \[\tilde{\phi} = \left( \mathrm{Id}_{\Sigma B} \vee \psi \right) \comp \phi,\]
    and let $\Sigma F \overset{\tilde{i}}{\hookrightarrow} E_\psi \rightarrow \Sigma B$ be the fibration obtained from the characteristic map $\psi\comp \rho$, with the canonical section $\tilde{s}$. Arguing as in (1), we obtain a homotopy equivalence $\Psi: E \rightarrow E_\psi$, such that $\Psi \comp \left( s \vee i \right) \simeq \left( \tilde{s} \vee \tilde{i} \right) \comp \left( \mathrm{Id}_{\Sigma B} \vee \psi\right)$. By Lemma~\ref{lemma:mapOfFibrations}, we then have 
    \[\psi_*\left\{ \mathrm{Id}_{\Sigma B}, \varphi \right\}_s = \left\{ \mathrm{Id}_{\Sigma B}, \psi\comp \varphi \right\}_{\tilde{s}} = \left\{ \mathrm{Id}_{\Sigma B}, \mathrm{Id}_{\Sigma F} \right\}_{\tilde{s}} = J[\varepsilon] -J[\psi\comp \rho]\]
    Applying $\varphi_*$ on both sides, we have from Lemma~\ref{J(epsilon)=0} that $\left\{ \mathrm{Id}_{\Sigma B}, \varphi \right\}_s = J[\phi \comp \varepsilon] -J[\rho]$. In particular, if $B$ is a suspension, it follows that $\left\{ \mathrm{Id}_{\Sigma B}, \varphi \right\}_s = -J[\rho]$.
\end{proof}

As a consequence, we have the following.
\begin{prop}\label{J-HomoInFibrewiseCompactification}
    Suppose $\Sigma F \hookrightarrow E \rightarrow \Sigma B$ is a fibration with a section $s$, and let $\rho: B \rightarrow Map(\Sigma F, \Sigma F)$ be the characteristic map. Then, $\Sigma \left\{ \mathrm{Id}_{\Sigma B}, \mathrm{Id}_{\Sigma F} \right\}_s$ is independent of the choice of section $s$, where $\Sigma : [\Sigma B \wedge F, \Sigma F] \rightarrow [\Sigma^2 B \wedge F, \Sigma^2 F]$ is the suspension homomorphism. In particular, 
    \[\Sigma \left\{ \mathrm{Id}_{\Sigma B}, \mathrm{Id}_{\Sigma F} \right\}_s = J[\varepsilon] - J[\xi],\]
    where $\varepsilon, \xi : B \rightarrow Map_\bullet(\Sigma^2 F, \Sigma^2 F)$ are respectively  given as 
    \[\varepsilon : b \mapsto \mathrm{Id}_{\Sigma^2 F}, \quad \xi : b \mapsto \Sigma \rho(b).\]
    Furthermore, if $B$ is a suspension, then $\Sigma \left\{ \mathrm{Id}_{\Sigma B}, \mathrm{Id}_{\Sigma F} \right\}_s = -J[\xi]$.
\end{prop}
\begin{proof}
    Suppose $\Sigma_f E$ denotes the total space of the fibration characterized by $\xi$. Then, $\Sigma^2 F \overset{j}{\hookrightarrow} \Sigma_f E \rightarrow \Sigma B$ is the fibrewise suspension, which admits the canonical section $s_\infty$ as defined in \eqref{eq:section_s_infty}. Now, consider the following diagram
    \[\begin{tikzcd}
        \Sigma B \wedge F \arrow[dashed]{r}{\left\{ \mathrm{Id}_{\Sigma B},  \mathrm{Id}_{\Sigma F} \right\}_s} &[5em] \Sigma F \arrow[hook]{r}{i} \arrow[hook]{d}{\iota} & E \arrow{r}{} \arrow[hook]{d}{\phi} & \Sigma B \arrow[equal]{d}{} \arrow[bend right=30]{l}[swap]{s} \\
        \Sigma B \wedge \Sigma F \arrow[dashed]{r}[swap]{\left\{ \mathrm{Id}_{\Sigma B}, \mathrm{Id}_{\Sigma^2 F} \right\}_{s_\infty}} & \Sigma^2 F \arrow[hook]{r}{j} & \Sigma_f E \arrow{r}{} & \Sigma B \arrow[bend left=30]{l}{s_\infty} 
    \end{tikzcd}\]
    where $\phi\comp i=j\comp \iota$ and $\phi \comp s \simeq s_{\infty}$. Then it follows that  
    \begin{align*}
        j_* \Sigma \left. \left\{ \mathrm{Id}_{\Sigma B}, \mathrm{Id}_{\Sigma F} \right\}_s\right|_{\Sigma B\wedge F} &= \phi_* i_* \left\{ \mathrm{Id}_{\Sigma B}, \mathrm{Id}_{\Sigma F} \right\}_s = \phi_* [s, i]\\
        &= [\phi \comp s, \phi \comp i] = [s_\infty, j \comp \iota] = j_*\left.\left\{ \mathrm{Id}_{\Sigma B}, \mathrm{Id}_{\Sigma^2 F} \right\}_{s_\infty}\right|_{\Sigma B\wedge F}.
    \end{align*}
    As $j_*$ is monic, we infer that $\Sigma \left\{ \mathrm{Id}_{\Sigma B}, \mathrm{Id}_{\Sigma F} \right\}_s = \left\{ \mathrm{Id}_{\Sigma B}, \mathrm{Id}_{\Sigma^2 F} \right\}_{s_\infty}$. Therefore, the proof follows from Theorem~\ref{BraceAndJ-homoInSuspensionOverSuspension}(1).
\end{proof}

Let us now prove that the brace product detects the homotopy type of the total space of the fibration.
\begin{thm}\label{thm:genBraceAndHomotopyType}
    For $r = 1, 2$, let $\Sigma F \overset{i_r}{\hookrightarrow} E_r \rightarrow  \Sigma B$ be two fibrations as above, given by characteristic maps $\rho_r$, with canonical sections $s_r$, respectively. Let $\varphi: \Sigma F \rightarrow \Sigma F$ be some base point preserving self-homotopy equivalence. Then the following are equivalent.
    \begin{enumerate}[(i)]
        \item $\varphi_* \left\{ \mathrm{Id}_{\Sigma B}, \mathrm{Id}_{\Sigma F} \right\}_{s_1} = \left\{ \mathrm{Id}_{\Sigma B}, \varphi \right\}_{s_2}$. 

        \item $J[\rho_1]=J[\rho_2]$.
        
        \item There exists a homotopy equivalence $\Phi : E_1 \rightarrow E_2$ satisfying 
        \[\Phi \comp \left( s_1 \vee i_1 \right) \simeq \left( s_2 \vee i_2 \right) \comp \left( \mathrm{Id}_{\Sigma B} \vee \varphi \right) = s_2 \vee \left( i_2 \comp \varphi \right).\]
    \end{enumerate}
\end{thm}
\begin{proof}\sloppy
    Note that the equivalence between (i) and (ii) follows from Theorem~\ref{BraceAndJ-homoInSuspensionOverSuspension} and Lemma~\ref{J(epsilon)=0}(2). Thus, it is enough to prove the equivalence between (i) and (iii).

    For $r = 1,2$, we have the attaching maps $\phi_r : CB \times F \cup B \times CF \rightarrow \Sigma B \vee \Sigma F$ as given in \eqref{eq:totalSpaceAttachinMapFormula}, and homotopy equivalences
    \[\Phi_r : E_r \rightarrow \tilde{E}_r = \left( CB \times CF \right) \cup_{\phi_r} \left( \Sigma B \vee \Sigma F \right),\]
    as in \eqref{eq:attachingMapEquivalenceDiagram}. We have, $\Phi_r \comp \left( s_r \vee i_r \right) \simeq j_r$, where $j_r : \Sigma B \vee \Sigma F \rightarrow \tilde{E}_r$ is the induced map. From the same diagram, by considering the inclusion $\Sigma F \hookrightarrow CB \times \Sigma F$, we also construct a homotopy inverse, say, $\Theta_r : \tilde{E}_r \rightarrow E_r$ of $\Phi_r$, such that $\Theta_2 \comp j_2 \simeq s_r \vee i_r$.
    
    Suppose $\varphi_* \left\{ \mathrm{Id}_{\Sigma B}, \mathrm{Id}_{\Sigma F} \right\}_{s_1} = \left\{ \mathrm{Id}_{\Sigma B}, \varphi \right\}_{s_2}$.
    It follows from Theorem~\ref{BraceAndJ-homoInSuspensionOverSuspension}(1) that 
    \[\varphi_* \left\{ \mathrm{Id}_{\Sigma B}, \mathrm{Id}_{\Sigma F} \right\}_{s_1} = \left\{ \mathrm{Id}_{\Sigma B}, \varphi \right\}_{s_2} \Rightarrow \varphi_* \left( J[\varepsilon] - J[\rho_1] \right) = J[\varphi \comp \varepsilon] - J[\rho_2].\]
    Let $\psi: \Sigma F \rightarrow \Sigma F$ be a choice of homotopy inverse so that $\psi \comp \varphi \simeq \mathrm{Id}_{\Sigma F} \simeq \varphi \comp \psi$. Let us denote $\tilde{\phi}_2 = \left( \mathrm{Id}_{\Sigma B} \vee \psi \right) \comp \phi_2$. In particular, the fibration obtained from the attaching map $\tilde{\phi}_2$ corresponds to the characteristic map $\psi \comp \rho_2$. From \eqref{eq:attachinMapDecomposition}, we now have
    \begin{align*}
        (\mathrm{Id}_{\Sigma B} \vee \varphi)_* [\phi_1 \comp \mu]
        &= (\mathrm{Id}_{\Sigma B} \vee \varphi)_* \iota_* \left( [\tilde{\rho}_1 \comp \mu] + [\epsilon \comp \mu] \right) + (\mathrm{Id}_{\Sigma B} \vee \varphi)_* [\omega \comp \mu] \\
        &= \iota_* \varphi_* \left( J[\rho_1] - J[\varepsilon] \right) + (\mathrm{Id}_{\Sigma B} \vee \varphi)_* [\omega \comp \mu] \\
        &= \iota_* \left( J[\rho_2] - J[\varphi \comp \varepsilon] \right) + (\mathrm{Id}_{\Sigma B} \vee \varphi)_*[\omega \comp \mu] \\
        &= \iota_* \varphi_* \left( J[\psi \comp \rho_2] - J[\varepsilon]\right) + (\mathrm{Id}_{\Sigma B} \vee \varphi)_*[\omega \comp \mu] \\
        &= (\mathrm{Id}_{\Sigma B} \vee \varphi)_* \iota_*\left( [\widetilde{\psi \comp \rho_2} \comp \mu] + [\epsilon \comp \mu]\right) + (\mathrm{Id}_{\Sigma B} \vee \varphi)_* [\omega \comp \mu] \\
        &= (\mathrm{Id}_{\Sigma B} \vee \varphi)_* [\tilde{\phi}_2 \comp \mu]
    \end{align*}
    As $\mu$ is a homotopy equivalence, we have 
    \[\phi_1 \simeq \tilde{\phi}_2 \Rightarrow \left( \mathrm{Id}_{\Sigma B} \vee \varphi \right) \comp \phi_1 \simeq \phi_2.\]
    We thus have the (homotopy) commutative diagram
    \[\begin{tikzcd}
        \left( CB \times F \right) \cup \left( B \times CF \right) \arrow[equal]{rd}{} \arrow[hook]{dd}{} \arrow{rr}{\phi_1} && \Sigma B \vee \Sigma F \arrow{rd}{\mathrm{Id}_{\Sigma B} \vee \varphi}[swap]{\simeq} \arrow{dd}[pos=.3]{j_1}\\
        & \left( CB \times F \right) \cup \left( B \times CF \right)  \arrow[crossing over]{rr}[pos=.3]{\phi_2} && \Sigma B \vee \Sigma F \arrow{dd}[swap]{j_2} \\
        CB \times CF \arrow{rr}{} \arrow[equal]{rd}{} && \tilde{E}_1 \arrow[dashed]{rd}{\Psi}[swap,pos=.7]{\simeq} \\
        & CB \times CF \arrow{rr}{} \ar[from=uu, hook, crossing over] && \tilde{E}_2 \\
    \end{tikzcd}\]
    As the pushouts are homotopy pushouts, we have an induced homotopy equivalence $\Psi: \tilde{E}_1 \rightarrow \tilde{E}_2$. By construction, $\Psi \comp j_1 \simeq \Phi_1 \comp \left( s_1 \vee i_1 \right)$. Set, $\Phi = \Theta_2 \comp \Psi \comp \Phi_1$. Thus, $\Phi: E_1 \rightarrow E_2$ is a homotopy equivalence. Also, we have
    \begin{align*}
        \Phi \comp \left( s_1 \vee i_1 \right) &\simeq \left( \Theta_2 \comp \Psi \comp \Phi_1 \right) \comp \left( s_1 \vee i_1 \right) \\
        &\simeq \Theta_2 \comp \Psi \comp j_1 \\
        &\simeq \Theta_2 \comp j_2 \comp \left( \mathrm{Id}_{\Sigma B} \vee \varphi \right) \\
        &\simeq \left( s_2 \vee i_2 \right) \comp \left( \mathrm{Id}_{\Sigma B} \vee \varphi \right) = s_2 \vee \left( i_2 \comp \varphi \right).
    \end{align*}

    Conversely, suppose we have a homotopy equivalence $\Phi : E_1 \rightarrow E_2$ satisfying $\Phi \comp \left( s_1 \vee i_1 \right) \simeq s_2 \vee \left( i_2 \comp \varphi \right) = \left( s_2 \vee i_2 \right) \comp \left( \mathrm{Id}_{\Sigma B} \vee \varphi \right)$. Then, it follows from Lemma~\ref{lemma:mapOfFibrations} that $\varphi_* \left\{ \mathrm{Id}_{\Sigma B}, \mathrm{Id}_{\Sigma F} \right\}_{s_1} = \left\{ \mathrm{Id}_{\Sigma B}, \varphi \right\}_{s_2}$.
\end{proof}

\subsection{Applications}
Theorem~\ref{BraceAndJ-homoInSuspensionOverSuspension}, specialized to the case of sphere bundle over a sphere, aids in demonstrating Lemma 1 from Milnor \cite{Milnor-OnWhiteheadJHomo}. To the best of the authors' knowledge, no proof of this result could be located in the existing literature.

\begin{lemma}[{\cite[Lemma 1]{Milnor-OnWhiteheadJHomo}}]\label{Milnor-lemma}
    Let $\xi$ be the oriented rank $q$ vector bundle over $S^n$ $(n\geq 2)$ corresponding to an element $\rho\in \pi_{n-1}(SO(q))$. Also, let $Th(\xi)$ be the Thom space associated with $\xi$. Then $Th(\xi)$ has a cell structure $D^{n+q}\cup_\Phi S^q$, where the attaching map $\Phi: S^{n+q-1}\to S^q$ represents $J(\rho)\in \pi_{n+q-1}(S^q)$.
\end{lemma}
\begin{proof}    
    Suppose $E$ is the $S^q$-bundle over $S^n$ associated with $\rho \in \pi_{n-1}(SO(q))$, equipped with the section at infinity $s$ as described in \eqref{eq:section_s_infty}. In this context, $E$ is simply the fibrewise compactification of the real oriented rank $q$ bundle $\xi$. Consequently, the Thom space $Th(\xi)$ can be identified with $E/s(S^n)$. Furthermore, as outlined in the proof of Theorem~\ref{BraceAndJ-homoInSuspensionOverSuspension}(1), the cell structure of $E$ is given by $D^{n+q} \cup_\phi S^n \vee S^q$, where $\phi: S^{n+q-1} \to S^n \vee S^q$ represents the element $\iota_*J(\rho) + [\omega]$. Here, $\iota: S^q \hookrightarrow S^n \vee S^q$ denotes the inclusion, and $\omega: S^{n+q-1} \to S^n \vee S^q$ is the Whitehead map. Note that, in the above expression for the attaching map $\phi$, the term $J[\epsilon]$ (see \eqref{eq:attachinMapDecomposition}) does not appear because the base $S^n$ $(n \geq 2)$ is a double suspension (by Lemma~\ref{J(epsilon)=0}). Therefore, the cell structure of the Thom space $E/s(S^n)$ is $D^{n+q}\cup_\Phi S^q$, where $\Phi:S^{n+q-1}\to S^q$ represents $(\text{proj}_2)_*(\iota_*J(\rho)+[\omega])$. However, since the composite $S^{n+q-1}\xrightarrow{\omega} S^n\vee S^q\xrightarrow{\text{proj}_2} S^q$ is null-homotopic, and the composite $S^{q}\xrightarrow{\iota} S^n\vee S^q\xrightarrow{\text{proj}_2} S^q$ is the identity, the attaching map $\Phi$ reduces to $J(\rho)$.
\end{proof}

We also found that Proposition 6.4 of Husem\"{o}ller \cite{Husemoller-FibreBundles} is not true. This, in particular, states that ``\textit{image of $J$-homomorphism is contained inside the image of suspension homomorphism of sphere}''. We demonstrate it in the following example.
\begin{eg}\label{eg:Husemollar_counter}
    Consider the $J$-homomorphism $J: \pi_3(SO(3))\to \pi_6(S^3)$. It is known (see top paragraph of p. 65 and Theorem 7.2 (i) of \cite{Toda-HomotopyGroupOfSpheres}) that this map is onto. But the suspension homomorphism $\Sigma: \pi_5(S^2)\cong \Z_2\to \pi_6(S^3)\cong \Z_{12}$ is clearly not surjective. Thus, the image of the $J$-homomorphism cannot be contained in the image of the suspension homomorphism $\Sigma$. 
\end{eg}

We propose the following rectification of \cite[Proposition 6.4]{Husemoller-FibreBundles}. Our proposed rectification is a direct corollary of Proposition~\ref{J-HomoInFibrewiseCompactification}.
\begin{prop}\label{HusemollerRectified}
    For $n, q \ge 2$, let $[\rho]$ be in the image of $\iota_* : \pi_{n-1}(SO(q)) \rightarrow \pi_{n-1}(SO(q + 1))$, induced by the inclusion $\iota : SO(q)\hookrightarrow SO(q+1)$. Then, $J[\rho] \in \pi_{n + q}(S^{q+1})$ is a suspension. In particular, $J[\rho] = - \Sigma \left\{ \mathrm{Id}_{S^n}, \mathrm{Id}_{S^q} \right\}_s$, where the brace product is taken in the sphere bundle $S^{q} \hookrightarrow E \rightarrow S^n$ characterized by $\rho$, which admits a section $s$.
\end{prop}
\begin{proof}
    Let $[\rho] = \iota_*[\xi]$ for some $[\xi] \in \pi_{n-1} (SO(q))$. Consider the bundles 
    $$E_\rho = D^n_+ \times S^q \cup_\rho D^n_- \times S^q\text{ and }E_\xi = D^n_+ \times S^{q-1} \cup_\xi D^n_- \times S^{q-1},$$ 
    where the actions are given by $\rho$ and $\xi$ respectively. Then, the fibrewise reduced suspension of $E_\xi$ is $E_\rho$. We have, $E_\rho = D^n_+ \times S^q \cup_\xi D^n_- \times S^q$, where $SO(q)$ acts on $S^q = \mathbb{R}^q \cup \left\{ \infty \right\}$ via $\xi$. Consequently, $E_\rho$ admits a section $s = s_\infty$, which maps to the $\infty$ point in each fibre. Now, by Theorem~\ref{BraceAndJ-homoInSuspensionOverSuspension}, we have $\left\{ \mathrm{Id}_{S^n}, \mathrm{Id}_{S^q} \right\}_s = -J[\xi]$. Taking  suspension on both sides, we have 
    \[\Sigma \left\{ \mathrm{Id}_{S^n} , \mathrm{Id}_{S^q} \right\}_s = -\Sigma J[\xi] = - J(\iota_*[\xi]) = -J[\rho],\]
    where the second equality follows from \cite[Thm 2]{Whitehead-OnTheHomotopy}. This concludes the proof.
\end{proof}

\begin{thm}\label{thm:sphereBundleRationallySplits}
    Consider a fibre bundle $S^q \hookrightarrow E \rightarrow S^n$ with structure group $SO(q + 1)$. If the fibration admits a homotopy section $s:S^n\to E$, then $E$ has the same rational homotopy type as $S^q \times S^n$.
\end{thm}
\begin{proof}
    It is enough to show that for some choice of homotopy section $s$, either the generalized brace $\left\{ \mathrm{Id}_{S^n_{\mathbb{Q}}}, \mathrm{Id}_{S^q_{\mathbb{Q}}} \right\}_{s_\Q} = 0$ (Proposition~\ref{splitting_in_suspension_over_suspension}), or the James brace product $\left\{, \right\}_{s_{\mathbb{Q}}}$ vanishes identically (Corollary~\ref{cor:JamesBraceVanishImpliesRationalEquiv}).

    When $q$ is odd, it was observed in Example~\ref{eg:oddSphereBundleJamesVanish} that the James brace product $\left\{ , \right\}_{s_{\mathbb{Q}}}$ always vanishes. When $q$ is even, the homotopy groups of $S^q_{\mathbb{Q}}$ are concentrated in degrees $q$ and $2q - 1$. If $n \ne q$, then the brace product $\left\{ \mathrm{Id}_{S^n}, \mathrm{Id}_{S^q} \right\}_{s_{\mathbb{Q}}} \in \pi_{n+q-1}(S^q_{\mathbb{Q}}) = 0$. Thus, we are left with the possibility that $q = n$ is an even number.

    Let us consider the fibre bundles $S^n \hookrightarrow E \rightarrow S^n$ admitting a section and with structure group $SO(n+1)$ for even $n$. These bundles are classified by elements of $\pi_{n-1}(SO(n+1))$. If $\rho\in \pi_{n-1}(SO(n+1))$ represents the bundle $E$ and $f$ is the degree two map on $S^n$ obtained by $(n-1)$th suspension of the degree $2$ map $z\mapsto z^2$ on $S^1$, then $f^*E$ is the bundle represented by the element $2\rho\in \pi_{n-1}(SO(n+1))$.
    Note that any degree two map $f:S^n\to S^n$ is a rational homotopy equivalence. Hence, using the five lemma, $E$ and the pull-back bundle $f^*E$ are rationally equivalent. 
    Thus, studying the rational homotopy types of the bundles $E$ and $f^*E$ are equivalent. 

    Now from the fibration $SO(n) \hookrightarrow SO(n+1) \rightarrow S^n$, we have the long exact sequence 
    \[\cdots \rightarrow \pi_n(SO(n+1))\rightarrow \pi_n (S^n) \rightarrow \pi_{n-1}(SO(n)) \rightarrow \pi_{n-1}(SO(n+1)) \rightarrow 0.\]
    Since $SO(n+1)$ is a Lie group, $\pi_n(SO(n+1))$ is a torsion subgroup for any even $n$ (see \cite[Cor 1, \S 3, Chap V]{Serre-Groupes}). But $\pi_n(S^n)$ being free, the map $\pi_n(SO(n+1))\rightarrow \pi_n (S^n)$ must be zero, and we then have the following short exact sequence. 
    \begin{equation}\label{eq:soNExactSequence}
        0 \rightarrow {\pi_n(S^n)} \xrightarrow{\partial} \pi_{n-1}(SO(n)) \xrightarrow{i_*} {\pi_{n-1}(SO(n+1))} \rightarrow 0.
    \end{equation}
    It is known (see the paragraph just before the \S 3 of \cite{Whitehead-FreudenthalThm}) that the image of the map
    \begin{align*}
        P:\pi_n(S^n) &\rightarrow  \pi_{2n-1}(S^n) \\ 
        \alpha &\mapsto \left[ \alpha, \mathrm{Id}_{S^n} \right]
    \end{align*}
    generates the kernel of the suspension homomorphism $\Sigma: \pi_{2n-1}(S^n) \rightarrow \pi_{2n}(S^{n+1})$. Since $n$ is even, $[\mathrm{Id}_{S^n},\mathrm{Id}_{S^n}]\neq 0$ is of infinite order 
    and hence using linearity, we have $P$ is injective.  Also, the Freudenthal suspension theorem implies that $\Sigma$ is onto.
    Therefore, we have the following exact sequence.
    \[0\rightarrow \pi_n(S^n)\xrightarrow{P}\pi_{2n-1}(S^n) \xrightarrow{\Sigma} \pi_{2n}(S^{n+1})\rightarrow 0\]
    
    \noindent The above exact sequence together with (\ref{eq:soNExactSequence}) fits into the following commutative diagram (see \cite[Diagram 1.2]{James-Whitehead}).
    \begin{equation}\label{eq:commDiagramExactSeq}
        \begin{tikzcd}
            0 \arrow{r}{} & \pi_n(S^n) \arrow{r}{\partial} \arrow[equal]{d}{} & \pi_{n-1}(SO(n)) \arrow{r}{\iota_*} \arrow{d}{-J} & \pi_{n-1}(SO(n+1)) \arrow{d}{J} \arrow{r}{} & 0 \\
            0 \arrow{r}{} & \pi_n(S^n) \arrow{r}[swap]{P} & \pi_{2n-1}(S^n) \arrow{r}[swap]{\Sigma} & \pi_{2n}(S^{n+1}) \arrow{r}{} & 0
        \end{tikzcd}
    \end{equation}
    Now $n$ being even, we can write $\pi_{2n-1}(S^n)= \mathbb{Z} \oplus G$, where $G$ is some torsion subgroup.
    Since $\iota_*$ is surjective, for $\rho\in \pi_{n-1}(SO(n+1))$, there exists $\xi \in \pi_{n-1}(SO(n))$ such that $\rho = \iota_* \xi$ and thus $2\rho =\iota_*(2\xi)$.     
    Suppose $-J\xi=(m,g)\in \Z\oplus G$. Then $-J(2\xi)=(2m,2g)$.
    It is well-known that for $n$ even, 
    \begin{equation}\label{eq:Pmap}
        P\left(\mathrm{Id}_{S^n} \right)=
    \begin{cases}
        (1,0) &\text{ for }n \neq 2, 4, 8\\
        (2,g_n) &\text{ for }n=2, 4,8,
    \end{cases}
    \end{equation}
    where $g_n$ is a torsion element \cite[p. 274]{Hatcher}. Clearly $g_2 = 0$, and for the exact values of $g_n$ for $n=4,8$ we refer to \cite[Thm 4.1]{Toda-SomeRelaionsInHomotopyGroups}.
    Let us now consider 
    $$\xi'=
    \begin{cases}
        2\xi-2m\partial(\mathrm{Id}_{S^n}) &\text{ for }n\neq 2, 4, 8\\
        2\xi-m\partial(\mathrm{Id}_{S^n}) &\text{ for }n=2, 4,8
    \end{cases}$$ 
    so that $\iota_*(\xi')=2\rho$, as $\iota_*\partial=0$. We observe that $-J\xi'$ is a torsion element. Indeed, using \eqref{eq:Pmap} and the commutativity of the first square in \eqref{eq:commDiagramExactSeq}, we get
    $$-J\xi'=
    \begin{cases}
        -J(2\xi)+2mJ\partial(\mathrm{Id}_{S^n}) =(2m,2g)-(2m,0)=(0,2g)&\text{ for }n\neq 2, 4, 8\\
        -J(2\xi)+mJ\partial(\mathrm{Id}_{S^n})=(2m,2g)-(2m,mg_n)=(0,2g-mg_n) &\text{ for }n=2,4,8.
    \end{cases}$$
    Suppose $S^{n-1} \hookrightarrow \tilde{E} \rightarrow S^n$ is the bundle corresponding to $\xi'\in \pi_{n-1}(SO(n))$. Let $S^n \hookrightarrow E' \rightarrow S^n$ be the fiberwise suspension of $\tilde{E}$ with the canonical section $s = s_{\infty}$. Then, $E'$ is isomorphic to the bundle $f^*E$ given by $2\rho$. Thus, by Theorem~\ref{BraceAndJ-homoInSuspensionOverSuspension} we have $\left\{ \mathrm{Id}_{S^n}, \mathrm{Id}_{S^n} \right\}_s = -J\xi'$, which is a torsion element.

    Lastly, consider the rationalized fibration $S^n_{\mathbb{Q}} \overset{i_{\mathbb{Q}}}{\rightarrow} f^*E_{\mathbb{Q}}\xrightarrow{p_{\mathbb{Q}}}S^n_{\mathbb{Q}}$ with section $s_{\mathbb{Q}}$. Then as in Diagram \ref{Diagram:rationalization}, we have 
    \[\left[s_\Q\comp \varphi, i_\Q\comp \varphi\right]= [\varphi \comp s, \varphi \comp i] = \varphi_*[s, i] = \varphi_* i_* \underbrace{\left\{ \mathrm{Id}_{S^n}, \mathrm{Id}_{S^q} \right\}_s}_{\textrm{torsion}} = 0.\]
    Hence, Lemma \ref{RationalGenWhiteAndWhitehead} implies $\left( i_{\mathbb{Q}} \right)_*\left\{ \mathrm{Id}_{S^n_{\mathbb{Q}}}, \mathrm{Id}_{S^n_{\mathbb{Q}}} \right\}_{s_{\mathbb{Q}}} = [s_{\mathbb{Q}}, i_{\mathbb{Q}}] =0$, which gives $S^n_{\mathbb{Q}} \times S^n_{\mathbb{Q}} \simeq f^*E_{\mathbb{Q}}$ by Proposition~\ref{splitting_in_suspension_over_suspension}. Since $E$ and $f^*E$ are rationally equivalent, $E$ and $S^n \times S^n$ have the same rational homotopy type \cite[Proposition 9.8]{Felix_rational}.\qedhere
\end{proof}

\begin{remark}
    In the above theorem, we restricted the structure group to $SO(q+1)$; however, the proof indicates that such a restriction is needed only in the case where $n = q$ is even.
\end{remark}

\begin{eg}
    For $n \equiv 1, 2 \mod 8$ and $n > 1$, there are precisely two $S^n$ bundles over $S^n$ with $SO(n+1)$ action, since $\pi_{n-1}(SO(n+1)) = \mathbb{Z}_2$. It follows from \cite[Eg. 6.4]{Adams-J(X)group} that $\textrm{Im}\left( J : \pi_{n-1}(SO(n+1)) \rightarrow \pi_{2n}(S^{n+1}) \right) = \mathbb{Z}_2$. Then, by \cite{James-Whitehead}, the nontrivial bundle is homotopically nontrivial as well. On the other hand, it follows from Theorem~\ref{thm:sphereBundleRationallySplits} that the total space of any such bundle is rationally equivalent to $S^n \times S^n$.
\end{eg}

Observe that in the above example, pulling back the nontrivial bundle, say, $E$ by a degree $2$ map $f: S^n \rightarrow S^n$ trivializes it, and consequently, the induced map $f^*E = S^n \times S^n \rightarrow E$ is a rational equivalence using a $5$-lemma argument. Next, we give an example where this argument does not work, and yet the total space of the bundle has the rational homotopy type of a product.

\begin{eg}
    Consider the bundles $S^{12} \hookrightarrow E \rightarrow S^{12}$ with $SO(13)$ action, which are classified by $\pi_{11}(SO(13)) = \mathbb{Z}$. Now, $J : \pi_{11}(SO(13)) \rightarrow \pi_{24}(S^{13})$ is given by the map $\mathbb{Z} \rightarrow \mathbb{Z}_{504}$. In particular, pick some $\rho \in \pi_{11}(SO(13))$ which is not mapped to $0$ under $J$. Consider the bundle $S^{12} \hookrightarrow E \rightarrow S^{12}$ corresponding to $\rho$. Again by \cite{James-Whitehead}, we have $E \not\simeq S^{12} \times S^{12}$. On the other hand, Theorem~\ref{thm:sphereBundleRationallySplits} implies that $E$ is rationally equivalent to $S^{12} \times S^{12}$.
\end{eg}

\appendix
\section{} \label{sec:appendix}
The following proposition is known to be Theorem 5.2 of Eckmann-Hilton \cite{Eckmann-Hilton} in the literature. But we could not find it in the form below, so for completeness, we present a proof of it. The proof technique is similar to Eckmann-Hilton.  
\begin{prop}[{\cite[Thm 5.2]{Eckmann-Hilton}}]\label{homotopy_equivalence}
    Suppose $F\overset{i}{\hookrightarrow} E\xrightarrow{p} B$ is a fibration. If the loop space fibration $\Om F\overset{\Om i}{\hookrightarrow} \Om E\xrightarrow{\Om p} \Om B$ has a homotopy section $\Gamma$, then the loop space fibration splits i.e., $\Om E\simeq \Om B\times \Om F$. Moreover, $\theta:\Om B\times \Om F\to \Om E$ defined by $\theta:=\mu_E \circ (\Gamma\times \Omega i)$ is a homotopy equivalence, where $\mu_E$ denotes the loop concatenation in $E$.
\end{prop}
\begin{proof}
    For the given fibration, we have the following exact sequence
    \begin{align*}
        0 \rightarrow [\Om E, \Om F]\xrightarrow{(\Om i)_*} [\Om E, \Om E] \xrightarrow{(\Om p)_*} [\Om E, \Om B] \rightarrow 0
    \end{align*}
    Suppose $\Bar{I}_X: \Om X\to \Om X$ denotes the loop inversion map for $X=E, B$. Then $\mu_X(\Bar{I}_X(x),x)\simeq e_X\simeq \mu_X(x, \Bar{I}_X(x))$ for $x\in \Om X$, where $\mu_X$ is the loop concatenation in $\Om X$ and $e_X$ is the constant loop in $\Om X$. Also, let $\Delta: E\to E\times E$ be the diagonal map. Then for $a\in \Om E$,
    \begin{align*}
        & (\Om p)_*((\mu_E\circ (\Bar{I}_E\circ \Gamma\circ \Om p) \times 1)\circ \Delta)(a)\\
        =\ & \mu_B (\Om p\circ \Bar{I}_E\circ\Gamma\circ \Om p(a), \Om p(a)) \text{ (as $(\Om p)_*$ is a group homomorphism)}\\
        =\ & \mu_B (\Bar{I}_B(\Om p\circ\Gamma (\Om p(a)), \Om p(a))) \text{ (as $(\Om p)_*$ is a group homomorphism)}\\
        \simeq\ & \mu_B (\Bar{I}_B(\Om p(a)), \Om p(a)) \text{ (as $\Om p\circ \Gamma\simeq I_{\Om B}$  is identity map)}\\
        \simeq\ & e_B, \text{ constant loop in $\Om B$.}
    \end{align*}
    Thus from the exact sequence, we have $\mu_E\circ ((\Bar{I}_E\circ \Gamma\circ \Om p) \times 1)\circ \Delta\in \text{Ker}((\Om p)_*)=\text{Im}((\Om i)_*)$ i.e., there exists (unique) $f:\Om E\to \Om F$ such that $(\Om i)_*(f)=\mu_E\circ ((\Bar{I}_E\circ \Gamma\circ \Om p) \times 1)\circ \Delta$. We will show that $\eta:=(\Om p\times f)\circ \Delta: \Omega E\to \Om B\times \Om F$ is the homotopy inverse of the given $\theta: \Om B\times \Om F\to \Om E$ i.e., $\theta\circ \eta \simeq 1_{\Om E}$ and $\eta\circ \theta \simeq 1_{\Om B\times \Om F}$. Now for $a\in \Om E$,
    \begin{align*}
        \theta\circ \eta(a)
        &= \mu_E\circ (\Gamma\times \Om i)(\Om p(a), f(a))\\
        &= \mu_E (\Gamma(\Om p(a)), \Om i(f(a)))\\
        &= \mu_E (\Gamma(\Om p(a)), \mu_E (\Bar{I}_E(\Gamma(\Om p(a))), a))\text{ (by definition of $f$)}\\
        &\simeq \mu_E (\mu_E(\Gamma(\Om p(a)),  \Bar{I}_E(\Gamma(\Om p(a)))), a))\text{ (by homotopy associativity of $\mu_E$)}\\
        &\simeq \mu_E(e_E,a)\\
        &\simeq a
    \end{align*}
    Thus $\theta\circ \eta (a)\simeq a$ for all $a\in \Om E$ i.e., $\theta\circ \eta\simeq 1_{\Om E}.$
    For $(b,d)\in \Om B\times \Om F$,
    \begin{align*}
        \Om p\circ \theta(b,d)
        &= \Om p(\mu_E(\Gamma(b), \Om i(d))) \text{ (by definition of $\theta$)}\\
        &= \mu_B(\Om p(\Gamma(b)), \Om p(\Om i(d))) \text{ (as $(\Om p)_*$ is a group homomorphism)}\\
        &= \mu_B(b, e_B)\\
        &\simeq b \quad
    \end{align*}
    Also,
    \begin{align*}
        (\Om i)_*(f\circ \theta)(b,d)
        &= (\Om i)_*(f)(\theta(b,d))\\
        &= \mu_E\circ ((\Bar{I}_E\circ \Gamma\circ \Om p) \times 1)\circ \Delta (\mu_E(\Gamma(b), \Om i(d))) \text{ (by definition of $\theta$ and $f$)}\\
        &= \mu_E(\Bar{I}_E(\Gamma( \Om p(\mu_E(\Gamma(b), \Om i(d))))), \mu_E(\Gamma(b), \Om i(d)))\\
        &= \mu_E(\Bar{I}_E(\Gamma(\mu_B (\Om p(\Gamma(b)), \Om p(\Om i(d))))), \mu_E(\Gamma(b), \Om i(d)))\\
        &= \mu_E(\Bar{I}_E(\Gamma(\mu_B (b, e_B))), \mu_E(\Gamma(b), \Om i(d)))\\
        &\simeq \mu_E(\Bar{I}_E(\Gamma(b)), \mu_E(\Gamma(b), \Om i(d)))\\
        &\simeq \mu_E(\mu_E(\Bar{I}_E(\Gamma(b)), \Gamma(b)), \Om i(d)))\\
        &\simeq \mu_E(e_E, \Om i(d))\\
        &\simeq \Om i(d)
    \end{align*}
    As $(\Om i)_*$ is injective, we have $f\circ \theta (b,d)=d$. Thus 
    \begin{align*}
        \eta\circ \theta(b,d)
        &=(\Om p\times f)\circ \Delta(\theta (b,d))\\
        &= (\Om p\circ \theta(b,d), f\circ \theta(b,d))\\
        &\simeq (b,d)\text{ for all }(b,d)\in \Om B\times \Om F
    \end{align*}
    Consequently $\eta\circ \theta\simeq 1_{\Om B\times \Om F}$. This completes the proof.
\end{proof}

\bibliographystyle{abbrv}

\end{document}